\documentclass[letterpaper,reqno,11pt]{amsart}

\usepackage[T1]{fontenc}
\usepackage[utf8]{inputenc}
\usepackage{lmodern}

\usepackage[margin=1.5in]{geometry}

\usepackage{amsmath,amsthm}
\allowdisplaybreaks[3]
\usepackage{mathtools}
\let\coloneq\coloneqq
\let\eqcolon\eqqcolon
\usepackage{amssymb}
\usepackage{dsfont}
\usepackage{mathrsfs}
\usepackage{esint}

\usepackage{enumitem}

\usepackage{hyperref}
\hypersetup{colorlinks=false, hypertexnames=false}
\usepackage[nocompress]{cite}

\counterwithin{equation}{section}

\theoremstyle{plain}
\newtheorem{thm}{Theorem}[section]
\newtheorem{prop}[thm]{Proposition}
\newtheorem{lemma}[thm]{Lemma}
\newtheorem{cor}[thm]{Corollary}

\theoremstyle{definition}
\newtheorem{defn}[thm]{Definition}

\theoremstyle{remark}
\newtheorem{remark}[thm]{Remark}

\newcommand{\les}{\lesssim}

\renewcommand{\le}{\leqslant}
\renewcommand{\ge}{\geqslant}

\newcommand{\rd}{\mathrm{d}}
\newcommand{\re}{\mathrm{e}}
\newcommand{\ri}{\mathrm{i}}
\newcommand{\rt}{\mathrm{t}}
\newcommand{\cD}{\mathscr{D}}
\newcommand{\cE}{\mathscr{E}}
\newcommand{\cG}{\mathscr{G}}
\newcommand{\cN}{\mathscr{N}}
\newcommand{\cR}{\mathscr{R}}
\newcommand{\cS}{\mathscr{S}}
\newcommand{\cT}{\mathscr{T}}
\newcommand{\NN}{\mathbb{N}}
\newcommand{\ZZ}{\mathbb{Z}}
\newcommand{\QQ}{\mathbb{Q}}
\newcommand{\RR}{\mathbb{R}}
\renewcommand{\SS}{\mathbb{S}}
\newcommand{\TT}{\mathbb{T}}
\newcommand{\eps}{\varepsilon}
\newcommand{\1}{\mathds{1}}

\newcommand{\id}{\mathrm{id}}

\DeclareMathOperator{\Law}{Law}
\DeclareMathOperator{\tr}{\mathrm{tr}}
\DeclareMathOperator{\divergence}{\mathrm{div}}
\DeclareMathOperator{\EE}{\mathbb{E}}
\DeclareMathOperator{\var}{\mathrm{var}}
\DeclareMathOperator{\cov}{\mathrm{cov}}

\renewcommand{\hat}{\widehat}
\renewcommand{\bar}{\overline}
\renewcommand{\tilde}{\widetilde}
\newcommand{\norm}[2]{\lVert #1\rVert_{#2}}
\newcommand{\innerprod}[3]{\langle #1,#2 \rangle_{#3}}

\newcommand{\WP}{\mathrm{WP}}

\title[Gaussian measures under NSE and heat flows are mutually singular]{Gaussian measures under the Navier--Stokes and heat flows are mutually singular}
\author{Niko Nikov \and Leonardo Tolomeo}

\address{
Niko Nikov\\ School of Mathematics\\
The University of Edinburgh\\
and The Maxwell Institute for the Mathematical Sciences\\ James Clerk Maxwell Building\\ The King's Buildings\\ Peter Guthrie Tait Road\\ Edinburgh\\
EH9 3FD\\
United Kingdom}
\email{n.a.nikov@sms.ed.ac.uk}

\address{
Leonardo Tolomeo\\ School of Mathematics\\
The University of Edinburgh\\
and The Maxwell Institute for the Mathematical Sciences\\ James Clerk Maxwell Building\\ The King's Buildings\\ Peter Guthrie Tait Road\\ Edinburgh\\
EH9 3FD\\
United Kingdom}
\email{l.tolomeo@ed.ac.uk}



\date{}

\begin{document}

\begin{abstract}
We consider the incompressible Navier--Stokes equations on the torus $\mathbb{T}^d$, $d \ge 2$, with translation invariant Gaussian random initial data. We show that, for the Gaussian measure formally given by $\mu_0 \sim \exp(-\frac12 \| u \|_{\dot H^\alpha}^2)$ on divergence-free fields, its nonlinear evolution $\nu_t$ and its linear evolution $\mu_t$ are mutually singular for almost every time. Moreover, the same holds for the nonlinear evolution of an arbitrary initial measure $\nu_0$. This is in sharp contrast to the available results for Hamiltonian PDEs, where Gaussian measures are commonly quasi-invariant.

We achieve this by showing that there exists a set $\mathscr G_t$ such that $\mu_t(\mathscr G_t) = 1$, but for any smooth solution of Navier--Stokes, the set of times $t$ such that $u(t) \in \mathscr G_t$ is null, which we deduce from a quantitative convexity bound for frequency-localised versions of the (typically infinite) quantity $\| \re^{-t \Delta} u(t) \|_{L^2}^2$.
\end{abstract}

\maketitle

\section{Introduction} \label{sec:introduction}

The incompressible Navier--Stokes equations
\begin{equation} \label{eqn:NSE} \tag{NSE}
	\begin{cases}
		& ( \partial_t - \Delta ) u + (u \cdot \nabla) u + \nabla p = 0, \\
		& u(0) = u_0, \quad \nabla \cdot u = 0
	\end{cases}
\end{equation}
are the basic model for the motion of a viscous fluid. Their mathematical theory goes back to Leray \cite{Ler34}, who constructed global weak solutions, and to Fujita and Kato \cite{FK64}, who developed the theory of strong (mild) solutions; see \cite{Kat84,KT01} for well-posedness in scaling-critical spaces and \cite{RRS16} for a modern account. In two dimensions, smooth solutions are global \cite{Lad58,LP59}, while in three dimensions global regularity is a famous open problem \cite{Fef06}. Since turbulent flows are naturally described statistically, it is of interest to study \eqref{eqn:NSE} with random initial data, that is, to study how probability measures on the space of initial data are transported by the flow. For fluid equations, this point of view goes back at least to the Gaussian (Gibbs-type) measures that are invariant for the two-dimensional Euler equations, constructed by Albeverio and Cruzeiro \cite{AC90} (see also \cite{Fla18}), and to the analysis of the two-dimensional Navier--Stokes equations driven by space-time white noise of Da Prato and Debussche \cite{DPD02}. A random-data Cauchy theory for \eqref{eqn:NSE} has been developed in \cite{ZF11,DC11,NPS13}.

In this paper, we study the evolution of Gaussian measures $\mu_0$ outside of equilibrium. This is motivated by the study of \emph{quasi-invariance} for Hamiltonian PDEs. This was initiated by Tzvetkov \cite{Tzv15} for the BBM equation, where he proved that the Gaussian measure formally given by
\begin{equation} \label{eqn:gamma s}
	\gamma_s \sim \exp\Big( - \frac12 \| u \|_{H^s}^2 \Big)
\end{equation}
is quasi-invariant under the flow of BBM, or in other words, $(\Phi_t)_{\#}\gamma_s \ll \gamma_s$. This was especially striking at the time, since the result is not covered by the classical theory of transformations of Gaussian measures of Cameron and Martin \cite{CM44}, Ramer \cite{Ram74} and Kusuoka \cite{Kus82,Kus83}, and gives an extension to PDEs of the results by Cruzeiro \cite{Cru83a,Cru83b}.

This success has been replicated for many dispersive and Hamiltonian models, including nonlinear Schr\"odinger equations \cite{OT17,OST18,OTT19,DT21,PTV20,Kne25a,Kne25b,FT19,FS22,FT25,ST25,TV25,ST25b,GLT23b, GLT22}, nonlinear wave equations \cite{OT20,GOTW22,FT26}, the periodic generalised KdV equation \cite{PTV23}, and the Benjamin--Ono--BBM equation \cite{GLT23a}.

 However, in the setting of parabolic equations, including \eqref{eqn:NSE}, one can easily check that quasi-invariance must fail. This is due to the fact that the solution of \eqref{eqn:NSE} immediately becomes smooth (and actually analytic) for any $t > 0$, up to the blow-up time, while the event
\[ E_\infty = \{ u \in C^\infty\} \]
satisfies $\gamma_s(E_\infty) = 0$. In fact, this problem is not specific to nonlinear PDEs, and it is a consequence of the regularising effect of the linear flow.
 In contrast, for dispersive PDEs, the main mechanism that allows for quasi-invariance is the fact that the linear flow leaves the Gaussian measure \emph{invariant}, since it acts on each Fourier coefficient by a rotation. From a technical point of view, this is also a fundamental ingredient in the commonly used techniques. More precisely, what the quasi-invariance results often show is that $(\Phi_t)_{\#}\gamma_s \ll (S(t))_\#\gamma_s$, where $S(t)$ denotes the linear flow. Therefore, we define
\[ \mu_0 \sim \exp\Big( - \frac12 \| u \|_{\dot H^\alpha}^2 \Big), \]
restricted to divergence-free vector fields, and
 ask the following natural question:
\begin{center}\vspace{3pt}
 Let $\mu_t = (\re^{t\Delta})_\# \mu_0$. Does it hold that $(\Phi_t)_{\#} \mu_0 \ll \mu_t$?\vspace{3pt}
\end{center}

Comparing the nonlinear evolution with the linear evolution naturally removes the main obstruction. Indeed, an easy application of Ramer's theorem \cite{Ram74} shows that the answer is positive for the Galerkin approximations
\begin{equation} \tag{NSE-$N$}
	\begin{cases}
		& ( \partial_t - \Delta ) u + P_{\le N} (P_{\le N} u \cdot \nabla) P_{\le N} u + P_{\le N} \nabla p = 0, \\
		& u(0) = u_0, \quad \nabla \cdot u = 0,
	\end{cases}
\end{equation}
where $P_{\le N}$ denotes the Fourier projector onto the frequencies $|n| \le N$.

Showing such a property is also related to quasi-Gaussianity, which has been studied by Coe, Hairer and the second author in \cite{CHT25}, who showed that the invariant measure of the two-dimensional \emph{stochastic} Navier--Stokes equations admits a density with respect to the natural Gaussian measure under certain conditions on the noise (building upon the work of Mattingly and Suidan \cite{MS05}, which showed the analogous result for hyperviscous Navier--Stokes). Such a result for \eqref{eqn:NSE} could provide similar information for invariant measures for the deterministic (forced) Navier--Stokes system, following the work of Flandoli and Romito \cite{FR01}, where it is shown that invariant measures for the forced three-dimensional Navier--Stokes equations concentrate on smooth data.

Unfortunately, our main result, Theorem~\ref{thm:main precise} below, shows that for \eqref{eqn:NSE} the answer is \emph{negative}, and it does not depend on the choice of the initial measure. More precisely, we show the following.
\begin{thm} \label{thm:main precise}
	Let $\alpha > d + 6$, and let $\frac d2 + 6 < \sigma < \alpha - \frac d2$. Let $\nu_0$ be a probability measure on $\dot H^\sigma(\TT^d)$, and let $\mu_0$ be a centred Gaussian measure on $\dot H^\sigma(\TT^d)$ with covariance $(-\Delta)^{-\alpha}$ restricted to divergence-free vector fields. Denote $\mu_t \coloneq (\re^{t\Delta})_\# \mu_0$.
	Define, for $t \ge 0$, the set
	\begin{equation} \label{eqn:WP defn}
		\begin{aligned}
			& \WP(t) \coloneq \bigl\{ u_0 \in \dot H^\sigma(\TT^d) \colon \text{there exists a unique mild solution} \\
			& \hskip 30pt \text{$u \in C([0,t]; \dot H^\sigma(\TT^d))$ of \eqref{eqn:NSE} with $u(0) = u_0$} \bigr\}.
		\end{aligned}
	\end{equation}
	Let $\Phi_t \colon \WP(t) \to \dot H^\sigma(\TT^d)$ be the solution map $u_0 \mapsto u(t)$.
	Then, for every $s_0 \ge 0$, the set
	\begin{equation*}
		\cN_{s_0}
		\coloneq \bigl\{ t > 0 : (\Phi_t)_\#( \1_{\WP(t)} \, \nu_0) \not\perp \mu_{s_0+t} \bigr\}
	\end{equation*}
	has Lebesgue measure zero.
\end{thm}

In particular, this shows that the two measures are singular for almost every time, where we need to restrict the push-forward to the set of data whose solution does not blow up before time $t$. We remark that the restrictions on $\sigma$ and on $\alpha$ are not substantial, and we expect the above result to hold up to the critical space $\dot H^{\frac d2 - 1}(\TT^d)$.

Of course, one could ask if this mechanism survives after (smooth) blow-up. Indeed, it is in principle possible that a high regularity norm of a solution to \eqref{eqn:NSE} blows up in finite time, but the flow still has a unique continuation in a weaker space. As it turns out, we also show that the same negative result holds for any extension of the flow. More precisely, we have the following.
\begin{cor} \label{cor:extension}
	 Let $\cE \subseteq \cD' = \cD'(\TT^d)$ be a Borel set of divergence-free periodic distributions, and let $\tilde\Phi_t \colon \cE \to \cE$, $t \ge 0$, be Borel maps with $\tilde\Phi_0 = \id$ such that, for all $s \ge 0$, $r > 0$ and $u_0 \in \cE$,
	\begin{equation} \label{eqn:consistency}
		\tilde\Phi_s(u_0) \in \WP(r)
		\quad \Longrightarrow \quad
		\tilde\Phi_{s + r}(u_0) = \Phi_r\bigl( \tilde\Phi_s(u_0) \bigr).
	\end{equation}
	That is, $\tilde\Phi$ continues as the smooth flow $\Phi$ as soon as it reaches a datum for which a smooth solution exists.
	Then, for every Borel probability measure $\nu_0$ on $\cD'$ with $\nu_0(\cE) = 1$,
	\begin{equation*}
		(\tilde\Phi_t)_\# \nu_0 \perp \mu_t
		\qquad \text{for almost every } t > 0.
	\end{equation*}
\end{cor}

\subsection{Heuristic for singularity and strategy of the proof} \label{subsec:strategy}
In order to understand if one can have $(\Phi_t)_{\#}\mu_0 \ll \mu_t$, we can formally compute the Radon--Nikodym density
\[ \frac{ \rd (\Phi_t)_{\#}\mu_0}{ \rd \mu_t},\]
following the heuristic of \cite{CT26}. If $\Phi_t$ is the flow of an ODE $\dot y = b(y)$ on $\RR^D$ and $\mu_0(\rd y) = Z^{-1} \re^{-E(y)} \, \rd y$, then Liouville's equation, solved along characteristics, gives
\begin{equation*}
	\frac{\rd (\Phi_t)_\# \mu_0}{\rd y}(y)
	= Z^{-1} \exp\Bigl( -E\bigl( \Phi_{-t}(y) \bigr) - \int_0^t (\divergence b)\bigl( \Phi_{-t'}(y) \bigr) \, \rd t' \Bigr).
\end{equation*}
Applying this (formally) to \eqref{eqn:NSE}, with $E(u) = \frac12 \norm{ u }{ \dot H^\alpha }^2$, we obtain
\begin{equation} \label{eqn:formal density}
	\frac{\rd (\Phi_t)_\# \mu_0}{\rd \mu_t}(u)
	= \exp\Bigl( \tfrac12 \norm{ \re^{-t\Delta} u }{ \dot H^\alpha }^2 - \tfrac12 \norm{ \Phi_{-t}(u) }{ \dot H^\alpha }^2 \Bigr).
\end{equation}
We note that the choice of $\mu_t$ solving the linear equation manifests itself in the fact that the divergence term cancels out between the two expressions.
We then write the exponent in \eqref{eqn:formal density} as
\begin{equation*}
	 \tfrac12 \norm{ \re^{-t\Delta} u }{ \dot H^\alpha }^2 - \tfrac12 \norm{ \Phi_{-t}(u) }{ \dot H^\alpha }^2
	= \int_0^t \innerprod{ \re^{-2\tau\Delta} \Phi_\tau(u_0) }{ B\bigl( \Phi_\tau(u_0) \bigr) }{ \dot H^\alpha } \, \rd \tau,
\end{equation*}
where $u_0 \coloneq \Phi_{-t}(u)$.
Since the equation is semi-linear (and in view of the conjectured quasi-Gaussianity), we can test this expression on the linear flow $\re^{t\Delta}$. Namely, we expect that if this density is well-defined, then this quantity will have finite moments when we replace the nonlinear flow $\Phi_t$ with the linear flow $\re^{t\Delta}$. In this case, we obtain
\[ \int_0^t \innerprod{ u_0 }{ \re^{-\tau\Delta} B\bigl( \re^{\tau\Delta} u_0 \bigr) }{ \dot H^\alpha } \, \rd \tau.\]
One can then check that indeed
\[ \EE_{u_0 \sim \mu_0} \Big| \int_0^t \innerprod{ u_0 }{ \re^{-\tau\Delta} B\bigl( \re^{\tau\Delta} u_0 \bigr) }{ \dot H^\alpha } \, \rd \tau \Big|^2 = \infty \]
for every $t > 0$. Specifically, this is due to the Gevrey regularity of the term $ B\bigl( \re^{\tau\Delta} u_0 \bigr)$. More precisely, due to the quadratic nonlinearity, for typical Gaussian data, one has that (for large $n$)
\[ \log \bigl| \widehat{B ( \re^{\tau\Delta} u_0 )}(n) \bigr| = \sup_{n_1 + n_2 = n} \bigl( -\tau (|n_1|^2 + |n_2|^2) \bigr) + O(\log |n|) = -\frac\tau2 |n|^2 + O(\log |n|). \]
In particular, the object $\re^{-\tau\Delta} B\bigl( \re^{\tau\Delta} u_0 \bigr)$ becomes ill-defined. This strongly suggests that the measure $(\Phi_t)_\# \mu_0$ is singular with respect to $\mu_t$.

Indeed, such a phenomenon was already observed by Coe and the second author in \cite{CT26}, who showed that for the cubic Szeg\H{o} equation, the evolution of the Gaussian measure $\gamma_s$ as in \eqref{eqn:gamma s} is singular with respect to $\gamma_s$ for all but countably many times, in a suitable range of regularities $s$.

In our case, however, we point out that the approach of \cite{CT26} does not work. There are two distinct reasons for this: the first is that the reference measure $\mu_t$ depends on time, which prevents us from using the abstract result \cite[Proposition 4.1]{CT26}. We remark that this obstruction is substantial, and we are not aware of an immediate tweak that would give an analogous condition in this setting. The other obstruction is that the exponent of the formal density
\[\tfrac12 \norm{ \re^{-t\Delta} u }{ \dot H^\alpha }^2 - \tfrac12 \norm{ \Phi_{-t}(u) }{ \dot H^\alpha }^2 \]
is not dominated by its semi-linear behaviour (with $\Phi_t$ replaced by $\re^{t\Delta}$ and $u$ replaced by $\re^{t\Delta} u_0$), but instead every iterate of its Picard expansion generates terms that diverge at a faster rate as the frequency size $|n| \to \infty$.

Overall, mutual singularity of measures under the evolution of nonlinear PDEs is not a well-understood topic, and apart from the aforementioned result \cite{CT26}, we are aware only of results for ODEs (see for instance \cite{OST18,STX20}, where absolute continuity is shown to fail) and for linear transport equations (see \cite{dSL26}); in the stochastic setting, see also \cite{HKN24} for singular SPDEs. These techniques are not applicable to this setting.

Therefore, inspired by the formula for the formal density, in order to obtain Theorem~\ref{thm:main precise}, we rely on studying the quantity $\norm{ \re^{-t\Delta} u }{ \dot H^\alpha }^2$ on the Fourier shell $|n|^2 = j$. As in \cite{CT26}, we observe that, formally, along solutions of \eqref{eqn:NSE} distributed according to $\mu_t$,
\[ \frac{ \rd^2}{ \rd t^2} \norm{ \re^{-t\Delta} P_j u(t) }{ \dot H^\alpha }^2 \gg \norm{ \re^{-t\Delta} P_j u(t) }{ \dot H^\alpha }^2, \]
where $P_j$ denotes the projection on the shell (see \eqref{eqn:Pj defn}). This suggests a route for a contradiction. Indeed, writing $f_j(s) \coloneq \norm{ \re^{-s\Delta} P_j u(s) }{ \dot H^\alpha }^2$, if one had a quantitative version of quasi-Gaussianity, then we could write
\begin{equation} \label{eqn:contradiction heu}
	\EE_{\mu_t} [f_j] - \EE_{\mu_{t-h}} [f_j]
	= h \, \EE_{\mu_{t-h}} [f_j'] + \int_{t-h}^{t} \int_{t-h}^{s'} \EE_{\mu_s} [f_j''] \, \rd s \, \rd s',
\end{equation}
and as long as $h$ is large enough, derive a contradiction from the fact that the last term is much larger than the others.

To make this argument rigorous, we proceed in the following way: we suppose by contradiction that singularity fails on a set of times of positive measure, and consider the set $\cT \subseteq (0,\infty)$ of times at which $u(t)$ is ``close enough'' to being Gaussian (in the sense of Proposition~\ref{prop:Gaussian}). By a Fubini argument, we can show that this set has positive measure for some initial datum, and therefore it has Lebesgue points. However, at these Lebesgue points, we can bound the integral on the right-hand side of \eqref{eqn:contradiction heu} from below, which yields a contradiction. In order to fully bound the error terms in the expression, and in order to bound the length of the time step $h$ from below, we also need to develop a local well-posedness theory for \eqref{eqn:NSE} in exponentially weighted spaces, see Section~\ref{sec:local}. In the proof, the quantities $f_j$ are replaced by the weighted energies $\re^{\lambda t j} \norm{ P_j u(t) }{ \dot H^s }^2$, with $\lambda < 2$ close to $2$ and $\frac d2 < s \le \sigma - 6$, and the heuristic $f_j'' \gg f_j$ is made rigorous by the convexity estimate of Proposition~\ref{prop:local theory}\,(ii).

\subsection{Notation} \label{subsec:notation}
We consider \eqref{eqn:NSE}, that is,
\begin{equation*}
	\begin{cases}
		& ( \partial_t - \Delta ) u + (u \cdot \nabla) u + \nabla p = 0, \\
		& u(0) = u_0, \quad \nabla \cdot u = 0
	\end{cases}
\end{equation*}
with periodic boundary conditions on a fixed box, i.e.\ posed for $u \colon \RR_{\ge 0} \times \TT^d \to \RR^d$ in $d \ge 2$. Here, $\TT = [0,2\pi)/\sim$ is the interval with endpoints identified, equipped with the normalised Lebesgue measure $(2\pi)^{-1}\, \rd x$, where the prefactor will be dropped from our notation when integrating over $\TT$. We use the convention
\begin{equation*}
	\hat f(n)
	= \int_{\TT^d} f(x) \re^{-\ri n \cdot x} \, \rd x
\end{equation*}
as the Fourier transform. As the mean $\int u$ is conserved by the flow of \eqref{eqn:NSE}, we assume $u$ (and $u_0$) are mean-zero, that is, $\hat u(0) = \int u = 0$. We will work with \eqref{eqn:NSE} in the form
\begin{equation} \label{eqn:LerayNSE}
	\begin{cases}
		& ( \partial_t - \Delta ) u = B(u,u), \\
		& u(0) = u_0, \quad \divergence( u_0) = 0,
	\end{cases}
\end{equation}
where $B(u,v) = - \Pi ( (u \cdot \nabla) v ) = - \Pi \divergence ( u \otimes v )$ (the last identity holding for divergence-free $u$), $B(u) \coloneq B(u,u)$, and $\Pi$ denotes the Leray projection onto divergence-free vector fields. This can also be defined
via the Fourier symbol
\begin{equation} \label{eqn:Leray proj defn}
	\hat{\Pi f}(n)
	= \Pi_n \hat f(n)
	\coloneq \hat f(n) - \frac{\hat f(n) \cdot n}{|n|^2} n \eqcolon \Bigl( 1 - \frac{n \otimes n}{|n|^2} \Bigr) \hat f(n).
\end{equation}
The matrix $\Pi_n$ is nothing but the projection onto the orthogonal complement of the subspace of $\RR^d$ spanned by $n$. Let us clarify what we mean by a solution to \eqref{eqn:NSE}.
\begin{defn}[\eqref{eqn:NSE} mild solutions] \label{defn:solution}
	 Let $\sigma > \frac d2 + 6$ be as in Theorem~\ref{thm:main precise}.
	A function $u : [0,T] \times \TT^d \to \RR^d$ is a \textit{mild solution of \eqref{eqn:NSE}} if it belongs to the space
	\begin{equation*}
		C( [0,T] ; \dot H^\sigma(\TT^d) ) \cap L^2( [0,T] ; \dot H^{\sigma + 1}(\TT^d) ) \eqcolon C_T \dot H_x^\sigma \cap L_T^2 \dot H_x^{\sigma + 1}
	\end{equation*}
	and satisfies the integral formulation of \eqref{eqn:LerayNSE}, that is,
	\begin{equation} \label{eqn:LerayNSE int'd}
		u(t)
		= \re^{t \Delta} u_0 - \int_0^t \re^{(t - t')\Delta} \bigl( \Pi \bigl( ( u(t') \cdot \nabla ) u(t') \bigr) \bigr) \, \rd t'
	\end{equation}
	for all $t \in [0,T]$, in the sense of distributions.
\end{defn}
 In principle, mild solutions can be defined in the same way for any $\sigma \ge \frac d2 - 1$, the scaling-critical exponent for \eqref{eqn:NSE}. However, if  $u \in C_T \dot H_x^\sigma \cap L_T^2 \dot H_x^{\sigma + 1}$ is a mild solution for some $\sigma \ge \frac d2 - 1$, then $u \in C( (0,T] ; C^\infty(\TT^d))$, by classical parabolic smoothing and uniqueness arguments (see e.g.\ \cite{FK64,Kat84,RRS16}). Hence the notion of mild solution does not depend on $\sigma$ away from the initial time, and the choice of $\sigma$ only restricts the class of admissible initial data. In what follows, it is convenient to fix $\sigma$ large, and we include this choice in Definition~\ref{defn:solution}.
 In their seminal work \cite{FK64}, Fujita and Kato showed local existence and uniqueness of solutions with data in the scaling-critical space $H^{\frac12}$ in dimension three; see also \cite{Kat84,KT01} for other critical spaces. For the subcritical spaces $\dot H^\sigma(\TT^d)$, $\sigma > \frac d2 - 1$, which are the ones used in this paper, local well-posedness, continuous dependence on the data and the blow-up alternative are classical (see e.g.\ \cite{RRS16} for $d = 3$). Global well-posedness in dimension $2$ goes back to Ladyzhenskaya \cite{Lad58} and Lions--Prodi \cite{LP59}, while in dimension $d \ge 3$ it is a famous open problem; for $d = 3$, it is one of the Clay Millennium Prize problems \cite{Fef06}.
 Due to the lack of global well-posedness, Theorem~\ref{thm:main precise} is conditional on the existence of mild solutions, through the set $\WP(t)$ of ``good'' data for which a solution exists at least up to time $t$, defined in \eqref{eqn:WP defn}, and the solution map $\Phi_t \colon \WP(t) \to \dot H^\sigma(\TT^d)$, $u_0 \mapsto u(t)$.

 We now describe the Gaussian measure $\mu_0$ of Theorem~\ref{thm:main precise} in more detail. Here, $\cD' = \cD'(\TT^d)$ denotes the space of periodic distributions. To ensure that $\mu_0$ is invariant under spatial translations, we ask $\mu_0$ to be a centred Gaussian measure on $\cD'$ whose covariance operator $\Pi\Sigma\Pi$ is a Fourier multiplier, where $\Pi$ is the Leray projection \eqref{eqn:Leray proj defn}. Formally,
\begin{equation}
	\text{``} \mu_0( \rd u)
	\propto \exp\bigl( -\tfrac12 \innerprod{ ( \Pi \Sigma \Pi )^{-1} u }{ u }{} \bigr) \, \rd u \text{''}.
\end{equation}
Moreover, in analogy with the Hamiltonian setting, we choose $\Sigma = (-\Delta)^{-\alpha}$, that is,
\[
	\widehat{\Sigma f}(n) = |n|^{-2\alpha} \hat f(n), \qquad n \ne 0,
\]
with $\alpha > \frac d2$. With this choice, the formal expression above reads $\mu_0(\rd u) \propto \exp( -\frac12 \norm{ u }{ \dot H^\alpha }^2 ) \, \rd u$ on divergence-free fields, and the condition $\alpha > \frac d2$ ensures that $\mu_0$-almost every $u$ belongs to $L^2(\TT^d)$.

We have that $\mu_0$ is the law of the random Fourier series
\begin{equation} \label{eqn:mu0 defn}
	\mu_0
	\coloneq \Law \Bigl( x \mapsto \sum_{n \ne 0} \frac{ \Pi_n G_n }{ |n|^\alpha } \re^{\ri n \cdot x} \Bigr),
\end{equation}
where $\Pi_n |n|^{-2\alpha}$ is the Fourier symbol of $\Pi \Sigma \Pi$, and the $G_n$ are complex standard normal random vectors, independent up to the condition $G_{-n}=\bar{G_n}$, which ensures that $\mu_0$ is a measure on real fields. We write $g_n = \Pi_n G_n$. See Subsection~\ref{subsec:wick} and Section~\ref{sec:gaussian} for the results on Gaussian measures that we use.
We write $S(t) = \re^{t\Delta}$ to denote the linear propagator for the equation $\partial_t u = \Delta u$. Solutions of the heat equation are instantaneously analytic, since the Fourier coefficients of $S(t) u_0$ decay like $\re^{-t|n|^2}$. The role of this decay in our argument is explained in Subsection~\ref{subsec:strategy}. To capture this decay, we define projectors $P_j$ whose Fourier supports partition $\ZZ^d$ and such that (see Lemma~\ref{lem:Pj})
\begin{equation*}
	S(t) P_j = \re^{-tj} P_j.
\end{equation*}
More specifically, we define
\begin{equation} \label{eqn:Pj defn}
	\begin{aligned}
		& \hat{P_j f}(n) \coloneq \1_{A_j} \hat f(n), \\
		& A_j \coloneq \bigl\{ n \in \ZZ^d \colon j \le |n|^2 < j + 1 \bigr\}.
	\end{aligned}
\end{equation}
For any $J \ge 1$, given parameters $1 < \lambda < 2$, $t_0 \ge 0$, $\frac d2 < s \le \sigma - 6$, we define $X_J^s(\lambda,t_0)$ as the closure of divergence-free smooth functions under the norm
\begin{equation} \label{eqn:X defn}
	\norm{ f }{X_J^s(\lambda,t_0)}
	= \Bigl( \sum_{j \le J} \re^{ \lambda t_0 j } \norm{ P_j f }{ \dot H^s }^2 + \re^{ \lambda t_0 J } \norm{ P_{>J} f }{ \dot H^s }^2 \Bigr)^{\frac12}.
\end{equation}
 Since all functions are assumed to be mean-zero, this is a norm.
We abuse notation and use the same formula \eqref{eqn:X defn} to define $\norm{ f }{X_J^s(\lambda,t_0)}$ for every mean-zero (possibly matrix-valued, or not divergence-free) function $f$.

The notation $A \lesssim B$ will refer to the existence of a constant $C$, which may depend on the fixed parameters $d$, $\sigma$, $s$ and $\alpha$, for which $A \le CB$, and similarly for $\gtrsim,\sim$. When $C=C(\vec p)$ depends on some parameter tuple $\vec p$, we may write $A \lesssim_{\vec p} B$ to avoid confusion.

\subsection{Further remarks}
\begin{remark}[The case $d = 2$] \label{rmk:d2}
Let $d = 2$, let $\cE = L^2_{\mathrm{div}}$ be the space of mean-zero, divergence-free fields in $L^2(\TT^2)$, which is a Borel subset of $\cD'$, and let $\tilde\Phi_t$ be the global flow of \eqref{eqn:NSE} on $L^2_{\mathrm{div}}$. For every $u_0 \in L^2_{\mathrm{div}}$, \eqref{eqn:NSE} has a unique global weak solution \cite{Lad58,LP59}, which is smooth for positive times (see e.g.\ \cite{FT89}), and smooth solutions are global. Hence, $\tilde\Phi_t$ is continuous on $L^2_{\mathrm{div}}$ and \eqref{eqn:consistency} holds. Corollary~\ref{cor:extension} then shows that $(\tilde\Phi_t)_\# \nu_0 \perp \mu_t$ for almost every $t > 0$, for every Borel probability measure $\nu_0$ concentrated on $L^2_{\mathrm{div}}$. Since $\alpha > \frac d2 = 1$, this applies in particular to $\nu_0 = \mu_0$.
\end{remark}
\begin{remark} \label{rmk:nu0=mu0}
 As observed above, $\nu_0$ and $\mu_0$ need not be related. Since $\alpha > \sigma + \frac d2$, we have $\mu_0(\dot H^\sigma) = 1$, and since $\mu_0$ is concentrated on divergence-free fields we may take $\nu_0 = \mu_0$. In this case, Theorem~\ref{thm:main precise} shows that, for almost every $t > 0$ such that $\mu_0(\WP(t)) > 0$, the Gaussian measure $\mu_0$ transported by the Navier--Stokes flow is singular with respect to the same measure transported by the heat flow. This is in sharp contrast with the quasi-invariance results for Hamiltonian equations discussed at the beginning of this section, where the measure transported by the nonlinear flow is equivalent to its transport by the linear flow (which, in that setting, coincides with the measure itself).
\end{remark}

\begin{remark}[Statistical inference for fluids] \label{rmk:Nickl}
In a recent series of works, Nickl and collaborators developed a rigorous theory of Bayesian inference and data assimilation for dissipative evolution equations, including the two-dimensional Navier--Stokes equations, with Gaussian process priors on the initial condition. Posterior consistency is established in \cite{NT24}, and Gaussian (Bernstein--von Mises type) approximations of the posterior distribution are obtained in \cite{Nic24} for a class of nonlinear parabolic equations, and in \cite{KN25} for the two-dimensional Navier--Stokes equations. In these approximation results, the posterior is compared with the law of the solution of a linear parabolic equation with Gaussian initial datum, and the distance between the two is measured in a Wasserstein distance.

Corollary~\ref{cor:extension} provides a justification for working with such weak metrics, and it is strong evidence that the same analysis in total variation cannot lead to positive results. More specifically, in the two-dimensional setting of these works, Remark~\ref{rmk:d2} shows that, for every initial distribution $\nu_0$ concentrated on mean-zero, divergence-free fields in $L^2(\TT^2)$ (for instance, a Gaussian prior), the law $\nu_t = (\tilde\Phi_t)_\# \nu_0$ of the solution at time $t$ and the heat-transported Gaussian measure $\mu_t$ of Theorem~\ref{thm:main precise} (for any $\alpha > 8$) satisfy $\norm{ \nu_t - \mu_t }{ \mathrm{TV} } \coloneq \sup_A | \nu_t(A) - \mu_t(A) | = 1$ for almost every $t > 0$, no matter how small. The mechanism behind this is the creation, by the quadratic nonlinearity, of high-frequency tails of size $\re^{-\frac t2 |n|^2}$ instead of $\re^{-t|n|^2}$. These tails are negligible for Wasserstein-type distances,
but they are detected by the total variation distance. We expect the same mechanism to be present when the heat flow is replaced by other linear parabolic evolutions, such as the ones considered in \cite{KN25} and \cite{Nic24}, although this is not covered by Theorem~\ref{thm:main precise}.
\end{remark}

\subsection{Organisation of the paper}
In Section~\ref{sec:prelim} we collect preliminary results.
 In Section~\ref{sec:gaussian} we prove the estimates that hold almost surely with respect to the Gaussian measures $\mu_t$. In Section~\ref{sec:local} we prove the deterministic estimates for the flow of \eqref{eqn:NSE} that we need, namely a local theory in the spaces $X_J^s(\lambda,t_0)$ and the quantitative convexity estimate. Finally, in Section~\ref{sec:sing} we combine these ingredients to prove a lower bound at Lebesgue points of the set of times at which the solution is ``too close to being Gaussian'', and we deduce Theorem~\ref{thm:main precise} and Corollary~\ref{cor:extension}.

\subsection*{Acknowledgements}

L.T.\ is funded by the European Union, via the ERC Starting Grant ``CritPDEsRand'', grant agreement no.\ 101219486. Views and opinions expressed are however those of the authors only and do not necessarily reflect those of the European Union or the European Research Council Executive Agency. Neither the European Union nor the granting authority can be held responsible for them.

\subsection*{Declaration on the use of AI}

A first draft of this paper, containing the main result and its proof, was written before any use of AI. We then used Claude Opus~5.5 (Anthropic), with the ``Extra'' reasoning setting, to proofread the manuscript and to help us improve the arguments and the exposition, including drafting some of the auxiliary proofs. We have verified every statement and every proof, and take sole responsibility for any errors that remain.

\section{Preliminaries} \label{sec:prelim}
In this section we collect the auxiliary results that are used throughout the paper. The projectors $P_j$ of Subsection~\ref{subsec:freq proj} diagonalise the heat semigroup exactly, which allows us to measure the decay of a function through the weights $\re^{\lambda t_0 j}$ in the spaces $X_J^s(\lambda,t_0)$. The key estimate of Subsection~\ref{subsec:basic est} is Lemma~\ref{lem:X prod est}: the nonlinearity loses at most a factor $\re^{(\frac\lambda4 + \eta) t_0 J}$ in these spaces, and this loss determines the time scale in Proposition~\ref{prop:local theory}\,(i). Finally, Subsection~\ref{subsec:wick} recalls the Wick--Isserlis theorem, which we use in the second-moment computations of Section~\ref{sec:gaussian}.

\subsection{On the frequency projectors \texorpdfstring{$P_j$}{P\_j}} \label{subsec:freq proj}
The following lemma shows that the $P_j$ diagonalise the heat semigroup.
\begin{lemma} \label{lem:Pj}
	Let $j \ge 0$. Then for every $t \ge 0$ and every $f \in L^2(\TT^d)$, we have $P_j S(t) f = \re^{-tj} P_j f$; in particular,
	\begin{equation}
		\norm{ P_j S(t) f }{L^2} = \re^{- t j } \norm{ P_j f }{L^2}.
	\end{equation}
\end{lemma}
\begin{proof}
	Since $|n|^2 \in \ZZ$, we have $A_j = \{ n \in \ZZ^d \colon |n|^2 = j \}$, so that $[P_j S(t) f]\,\hat{\phantom{i}}(n) = \1_{A_j}(n) \re^{-t|n|^2} \hat f(n) = \re^{-tj} \widehat{P_j f}(n)$ for every $n \in \ZZ^d$.
\end{proof}
In the proof of Theorem~\ref{thm:main precise}, it will be necessary to have some information on the size $|A_j|$ of the frequency support of $P_j$. Since $|n|^2 \in \ZZ$, $A_j$ is the set of lattice points on the sphere of radius $\sqrt j$, and $|A_j| = r_d(j)$ is the number of representations of $j$ as a sum of $d$ squares. While it would be tempting to use the volume heuristic
\begin{equation*}
	|A_j|
	\sim j^{\frac{d-1}2} \bigl( \sqrt{j + 1} - \sqrt j \bigr)
	\sim j^{\frac{d-2}2},
\end{equation*}
this only holds on average, and may fail for individual $j$: for instance, $|A_j| = 0$ for infinitely many $j$ when $d = 2, 3$ (namely for $j \equiv 3 \ (\mathrm{mod}\ 4)$, respectively $j = 4^a(8b+7)$). The sets $A_j$ are also \textit{too thin} for lattice-point counting in balls to be useful: for example, in $d = 3$, the best known bound for the error term in the approximation
\begin{equation*}
	\# \bigl\{ n \in \ZZ^3 \colon |n|^2 \le x \bigr\}
	= \tfrac{4\pi}3 x^{\frac32} + R_3(x)
\end{equation*}
is $R_3(x) = O_\eps( x^{\frac{21}{32}+\eps} )$ \cite{Hea99}, which is larger than the volume of $ \bigl\{ \xi \in \RR^3 \colon j \le |\xi|^2 < j + 1 \bigr\} \sim j^\frac12$. The problem is worse in $d = 2$ \cite{Hux03}. This will not be an issue in the present article, as we will only require that the sequence $(|A_j|)$ is unbounded. In fact, we have the following.
\begin{lemma} \label{lem:Aj count}
	Let $d \ge 2$, and let $A_j$ be as in \eqref{eqn:Pj defn}. Then
	\begin{equation} \label{eqn:Aj count}
		\limsup_{j \to \infty} |A_j|
		= \infty,
	\end{equation}
	and so there exists a strictly increasing sequence $(j_k)_{k \in \NN}$ such that $|A_{j_k}| \ge 1$ and $\sum_k |A_{j_k}|^{-1} < \infty$ (for instance, with $|A_{j_k}| \ge 2^k$).
\end{lemma}
\begin{proof}
	Throughout this proof, we borrow some classical notation from analytic number theory concerning representations of integers as sums of squares. Given a positive integer $m$, let $r_d(m)$ be the number of ways in which $m$ can be represented as the sum of $d$ squares, and let
	\begin{equation} \label{eqn:S=sum r}
		S_d(x)
		\coloneq \sum_{m \le x} r_d(m);
	\end{equation}
	note that $|A_j| = r_d(j)$, since $|n|^2 \in \ZZ$. \\
	\indent The case $d \ge 3$ is fairly innocuous. Indeed, $S_d$ is equal to the number of lattice points contained in a ball of radius $\sqrt x$, and so asymptotically as $x \to \infty$,
	\begin{equation*}
		S_d(x) \sim \mathrm{Vol}(B(0,\sqrt x)) \sim x^{\frac d2}.
	\end{equation*}
	In particular, from \eqref{eqn:S=sum r}, we have
	\begin{equation*}
		\max_{m \le x} r_d(m) \ge \frac{1}{x} \sum_{m \le x} r_d(m)
		\gtrsim x^{\frac d2 - 1} \to \infty,
	\end{equation*}
	and so we must have that
	\begin{equation*}
		\limsup_{j \to \infty} |A_j|
		= \infty.
	\end{equation*}
	In $d = 2$, the result follows from some elementary number theory. Consider the positive integers $m$ of the form
	\begin{equation*}
		m
		= \prod_{ p \equiv 1\, (\mathrm{mod}\,4) } p^{v_p(m)},
	\end{equation*}
	where the product runs over the primes $p \equiv 1 \pmod 4$. By the two-squares theorem (see \cite{Gro85}), we have
	\begin{equation*}
		r_2(m)
		= 4 \prod_{p \equiv 1\, (\mathrm{mod}\,4)} (v_p(m) + 1).
	\end{equation*}
	In particular, $r_2(5^k) = 4(k + 1)$ tends to infinity. This completes the proof.

\end{proof}

\subsection{Basic lemmas} \label{subsec:basic est}
\begin{lemma}[Basic estimates] \label{lem:basics}
	We collect some useful basic estimates. Throughout, $f,g,h \in C^\infty$ have mean zero, and $\norm{ f }{ \dot W^{s,p} } \coloneq \norm{ |\nabla|^s f }{ L^p }$, so that $\dot H^s = \dot W^{s,2}$.
	\begin{enumerate}[label=(\roman*)]
		\item \emph{(\eqref{eqn:NSE} trilinear form estimates)} Let $\frac1{p_1} + \frac1{p_2} + \frac1{p_3} = 1$, and assume that $1 < p_1 < \infty$ or that $\nabla \cdot f = 0$. Then
		\begin{equation} \label{eqn:NSE Hoelder}
			\bigl| \innerprod{f}{B(g,h)}{} \bigr|
			\lesssim \norm{f}{L^{p_1}} \norm{g}{L^{p_2}} \norm{h}{\dot W^{1,p_3}}.
		\end{equation}
		\item \emph{(Interpolation)} Let $1 < p, p_1, p_2 < \infty$, $0 \le \theta \le 1$, $s = \theta s_1 + (1 - \theta) s_2$ and $\frac1p = \frac\theta{p_1} + \frac{1 - \theta}{p_2}$. Then
		\begin{equation} \label{eqn:interpolation}
			\norm{ f }{ \dot W^{s,p} }
			\lesssim \norm{ f }{ \dot W^{s_1,p_1} }^\theta \norm{ f }{ \dot W^{s_2,p_2} }^{1 - \theta}.
		\end{equation}
	\end{enumerate}
\end{lemma}
\begin{proof}

Item (i) follows from H\"older's inequality, since $\innerprod{f}{B(g,h)}{} = - \innerprod{\Pi f}{(g \cdot \nabla) h}{}$, where $\Pi f = f$ if $\nabla \cdot f = 0$, while for $1 < p_1 < \infty$ the Leray projection is bounded on mean-zero functions in $L^{p_1}(\TT^d)$ by the Mihlin--H\"ormander multiplier theorem and transference, see \cite[Sections~6.2.3 and~4.3.2]{Gra14a} (see also \cite[Section~5.1.4]{Gra14a} for the Riesz transforms). Item (ii) follows from complex interpolation of homogeneous Sobolev spaces, see \cite[Chapters~4 and~6]{BL76}. In the case $p = p_1 = p_2 = 2$, which is the only one used in this paper, (ii) is simply H\"older's inequality in Fourier space.
\end{proof}
For $j\in \NN$, we define the set of indices
\begin{equation} \label{eqn:LambdaSetDef}
	\Lambda(j) = \{ (k,\ell) \in \NN^2: \sqrt{k+1} + \sqrt{\ell + 1} \ge \sqrt{j}\}.
\end{equation}
We have the following.
\begin{lemma} \label{lem:Lambda}
	Let $k,\ell,j \in \NN$. Then
	\begin{enumerate}[label=(\roman*)]
		\item If $(k,\ell) \in \Lambda(j)$, then
		\begin{equation} \label{eqn:kljinequality}
			k + \ell \ge \tfrac12 j - 2, \quad \tfrac12 k + \ell \ge \tfrac13 j - \tfrac32 \quad \text{and} \quad k + \tfrac12 \ell \ge \tfrac13 j - \tfrac32.
		\end{equation}
		\item Suppose that $(k,\ell)\not\in \Lambda(j)$. Then, for every $f,g \in \dot H^s(\TT^d)$, we have that
		\begin{equation} \label{eqn:FourierSupport}
			P_{j}(P_kf \otimes P_\ell g) = 0.
		\end{equation}
	\end{enumerate}
\end{lemma}
\begin{proof}
	We start from \eqref{eqn:kljinequality}. From the Cauchy--Schwarz inequality in the form
	\begin{equation*}
		1 \cdot a + 1 \cdot b
		\le ( a^2 + b^2 )^{\frac12} ( 1 + 1 )^{\frac12}
	\end{equation*}
	with $a = \sqrt{k+1}$ and $b = \sqrt{\ell+1}$, we have that
	\begin{equation*}
		k + \ell \ge \tfrac12(\sqrt{k+1} + \sqrt{\ell + 1})^2 - 2 \ge \tfrac12 j - 2.
	\end{equation*}
	The second inequality in \eqref{eqn:kljinequality} follows from a similar argument with the Cauchy--Schwarz inequality applied to $2 \cdot \frac12 a + \sqrt2 \cdot \frac1{\sqrt2} b$, and the third one follows from the second one, since $\Lambda(j)$ is symmetric in $(k,\ell)$. We now move to \eqref{eqn:FourierSupport}.
	We have
	\begin{equation*}
		\begin{aligned}
		& [P_{j}(P_kf \otimes P_\ell g) ] \widehat{\phantom{I}}(n) \\
		& \qquad = \sum_{n_1+n_2 = n}\1_{\{j \le |n|^2 < j+1\}} \1_{\{k \le |n_1|^2 < k+1\}}\1_{\{\ell \le |n_2|^2 < \ell+1\}} \hat{f}(n_1) \otimes \hat{g}(n_2).
		\end{aligned}
	\end{equation*}
	We now claim that, if $(k,\ell) \notin \Lambda(j)$, then
	\begin{equation*}
		\1_{\{j \le |n|^2 < j+1\}} \1_{\{k \le |n_1|^2 < k+1\}}\1_{\{\ell \le |n_2|^2 < \ell+1\}} = 0.
	\end{equation*}
	Indeed, suppose that $n,n_1,n_2$ are such that the product above is not $0$. Then, from $n_1 + n_2 = n$, we deduce that
	\begin{equation*}
		\sqrt{j} \le |n|
		\le |n_1|+|n_2|
		< \sqrt{k+1} + \sqrt{\ell+1},
	\end{equation*}
	i.e.\ $(k,\ell) \in \Lambda(j)$, a contradiction. This proves \eqref{eqn:FourierSupport}.
\end{proof}
\begin{lemma}[Two bilinear estimates in $X_J^s(\lambda,t_0)$] \label{lem:X prod est}
	Fix $\lambda$, $t_0 > 0$, $s$, $J$, defining $\norm{\,\cdot\,}{X_J^s(\lambda,t_0)}$. For every $\eta > 0$, and for every $f,g$ smooth, we have the bounds
	\begin{align}
		\label{eqn:X prod est low freq}
		& \norm{ (f \cdot \nabla) P_{\le J} g }{ X_J^s(\lambda,t_0) }
		\lesssim \re^{(\frac\lambda4 + \eta) t_0 J} \norm{ f }{X_J^s(\lambda,t_0)} \norm{ P_{\le J} g }{X_J^s(\lambda,t_0)}, \\
		\label{eqn:X prod est high freq}
		& \norm{ f \otimes P_{>J} g }{ X_J^s(\lambda,t_0) }
		\lesssim \norm{ f }{X_J^s(\lambda,t_0)} \norm{ P_{>J} g }{X_J^s(\lambda,t_0)},
	\end{align}
	where the implicit constant can depend on $d$, $\eta,t_0,s$ but not on $f,g, \lambda$ or $J \ge 1$. Moreover, the implicit constant is bounded uniformly for $t_0$ in compact subsets of $(0,\infty)$.
\end{lemma}
\begin{proof}
	We first prove \eqref{eqn:X prod est low freq}. We write $P_{\le J} g = \sum_{j \le J} P_j g$. By \eqref{eqn:FourierSupport} and the triangle inequality, we have
	\begin{align*}
		& \norm{ (f \cdot \nabla) P_{\le J} g }{X_J^s(\lambda,t_0)}^2 \\
		& = \sum_{j \le J} \re^{\lambda t_0 j} \norm{ P_j \bigl( (f \cdot \nabla) P_{\le J} g \bigr) }{\dot H^s}^2
		+ \re^{\lambda t_0 J} \norm{ P_{>J} \bigl( (f \cdot \nabla) P_{\le J} g \bigr) }{\dot H^s}^2 \\
		& \le \sum_{j \le J} \re^{\lambda t_0 j} \Bigl( \sum_{ \substack{ k,\ell \le J \\ (k,\ell) \in \Lambda(j) } } \norm{ P_j \bigl( (P_k f \cdot \nabla) P_\ell g \bigr) }{\dot H^s}
		+ \sum_{\ell \le J} \norm{ P_j \bigl( (P_{>J} f \cdot \nabla) P_\ell g \bigr) }{\dot H^s} \Bigr)^2 \\
		& \hphantom{\le \sum_{j \le J}\ } + \re^{\lambda t_0 J} \Bigl( \sum_{ \substack{ k,\ell \le J \\ (k,\ell) \in \Lambda(J)} } \norm{ P_{>J} \bigl( (P_k f \cdot \nabla) P_\ell g \bigr) }{\dot H^s}
		+ \sum_{\ell \le J} \norm{ P_{>J} \bigl( (P_{>J} f \cdot \nabla) P_\ell g \bigr) }{\dot H^s} \Bigr)^2 \\
		& \eqcolon \sum_{j \le J} \re^{\lambda t_0 j} \bigl( \mathrm{I.i} + \mathrm{I.ii} \bigr)^2
		+ \re^{\lambda t_0 J} \bigl( \mathrm{II.i} + \mathrm{II.ii} \bigr)^2.
	\end{align*}
	We now bound the above term by term.
	\begin{align*}
		\mathrm{I.i}
		& = \sum_{ \substack{ k,\ell \le J \\ (k,\ell) \in \Lambda(j) } } \norm{ P_j \bigl( (P_k f \cdot \nabla) P_\ell g \bigr) }{\dot H^s} \\
		& \lesssim \sum_{ \substack{ k,\ell \le J \\ (k,\ell) \in \Lambda(j) } } \norm{ P_k f }{ \dot H^s } \norm{ P_\ell g }{ \dot H^{s + 1}} \\
		& \lesssim \sum_{ \substack{ k,\ell \le J \\ (k,\ell) \in \Lambda(j) } } \sqrt \ell \re^{ -\frac\lambda2 t_0 ( k + \ell ) } \norm{ f }{ X_J^s(\lambda,t_0) } \norm{ P_{\le J} g }{ X_J^s(\lambda,t_0) } \\
		& \lesssim J^\frac52 \re^{ - \frac\lambda4 t_0 j } \norm{ f }{ X_J^s(\lambda,t_0) } \norm{ P_{\le J} g }{ X_J^s(\lambda,t_0) },
	\end{align*}
	where we used \eqref{eqn:kljinequality} in the last inequality. Next,
	\begin{align*}
		\mathrm{I.ii}
		& = \sum_{ \ell \le J } \norm{ P_j \bigl( ( P_{>J} f \cdot \nabla ) P_\ell g \bigr) }{ \dot H^s } \\
		& \lesssim \sum_{ \ell \le J } \sqrt \ell \re^{ -\frac\lambda2 t_0 (\ell + J) } \norm{ f }{ X_J^s(\lambda,t_0) } \norm{ P_{\le J} g }{ X_J^s(\lambda,t_0) } \\
		& \lesssim \re^{ - \frac\lambda2 t_0 J } \norm{ f }{ X_J^s(\lambda,t_0) } \norm{ P_{\le J} g }{ X_J^s(\lambda,t_0) }.
	\end{align*}
	Proceeding analogously for $\mathrm{II.i}$ and $\mathrm{II.ii}$, we obtain
	\begin{equation*}
		\mathrm{II.i} \les J^\frac52 \re^{ - \frac\lambda4 t_0 J } \norm{ f }{ X_J^s(\lambda,t_0) } \norm{ P_{\le J} g }{ X_J^s(\lambda,t_0) }, \quad \mathrm{II.ii} \les \re^{ - \frac\lambda2 t_0 J } \norm{ f }{ X_J^s(\lambda,t_0) } \norm{ P_{\le J} g }{ X_J^s(\lambda,t_0) }.
	\end{equation*}
	Therefore,
	\begin{align*}
		\norm{ (f \cdot \nabla) P_{\le J} g }{ X_J^s(\lambda,t_0) }^2 &\les \sum_{j\le J} \big(\re^{\lambda t_0 j} J^5 \re^{-\frac \lambda2 t_0 j} + \re^{\lambda t_0 (j-J)}\big) \norm{ f }{ X_J^s(\lambda,t_0) }^2 \norm{ P_{\le J} g }{ X_J^s(\lambda,t_0) }^2 \\
		&\les J^5 \re^{\frac\lambda2 t_0 J}  \norm{ f }{ X_J^s(\lambda,t_0) }^2 \norm{ P_{\le J} g }{ X_J^s(\lambda,t_0) }^2 \\
		&\les \re^{2(\frac\lambda4 + \eta) t_0 J} \norm{ f }{X_J^s(\lambda,t_0)}^2 \norm{ P_{\le J} g }{X_J^s(\lambda,t_0)}^2,
	\end{align*}
	where the contribution of $\mathrm{II.i}$ and $\mathrm{II.ii}$ is bounded in the same way. Taking square roots concludes the proof of \eqref{eqn:X prod est low freq}.

	To prove \eqref{eqn:X prod est high freq}, we proceed analogously. Here,
	\begin{align*}
		& \norm{ f \otimes P_{>J} g }{ X_J^s(\lambda,t_0) }^2 \\
		& = \sum_{j \le J} \re^{\lambda t_0 j} \norm{ P_j \bigl( f \otimes P_{>J} g \bigr) }{ \dot H^s }^2
		+ \re^{\lambda t_0 J} \norm{ P_{>J} \bigl( f \otimes P_{>J} g \bigr) }{ \dot H^s }^2 \\
		& \le \sum_{j \le J} \re^{\lambda t_0 j} \Bigl( \sum_{k \le J} \norm{ P_j \bigl( P_k f \otimes P_{>J} g \bigr) }{ \dot H^s } + \norm{ P_j \bigl( P_{>J} f \otimes P_{>J} g \bigr) }{ \dot H^s } \Bigr)^2 \\
		& \hskip 30pt + \re^{\lambda t_0 J} \Bigl( \sum_{ k \le J } \norm{ P_{>J} \bigl( P_k f \otimes P_{>J} g \bigr) }{ \dot H^s } + \norm{ P_{>J} \bigl( P_{>J} f \otimes P_{>J} g \bigr) }{ \dot H^s } \Bigr)^2 \\
		& \eqcolon \sum_{j \le J} \re^{\lambda t_0 j} \bigl( \mathrm{III.i} + \mathrm{III.ii} \bigr)^2 + \re^{\lambda t_0 J} \bigl( \mathrm{IV.i} + \mathrm{IV.ii} \bigr)^2.
	\end{align*}
	Again we proceed term by term. As before
	\begin{align*}
		\mathrm{III.i}
		& \lesssim \sum_{k \le J} \re^{ - \frac\lambda2 t_0 ( k + J ) } \norm{ f }{ X_J^s(\lambda,t_0) } \norm{ P_{>J} g }{ X_J^s(\lambda,t_0) } \\
		& \lesssim \re^{ - \frac\lambda2 t_0 J } \norm{ f }{ X_J^s(\lambda,t_0) } \norm{ P_{>J} g }{ X_J^s(\lambda,t_0) }.
	\end{align*}
	Similarly,
	\begin{equation*}
		\mathrm{III.ii}
		\lesssim \re^{ - \lambda t_0 J } \norm{ f }{ X_J^s(\lambda,t_0) } \norm{ P_{>J} g }{ X_J^s(\lambda,t_0) },
	\end{equation*}
	and the terms $\mathrm{IV.i}$ and $\mathrm{IV.ii}$ have the same bounds respectively. Putting the above together yields \eqref{eqn:X prod est high freq} and completes the proof.
\end{proof}
As a consequence of Lemma~\ref{lem:X prod est}, we obtain the following bound on the nonlinearity in the spaces $X_J^s(\lambda,t)$, which will be used in Section~\ref{sec:sing} to control the error term $\re^{-\delta tJ} \norm{ B(u(t)) }{ X_J^s(\lambda,t) }^2$ in Proposition~\ref{prop:local theory}\,(ii).
\begin{lemma} \label{lem:Z>E}
	Let $\tau > 0$ and $0 < t_1 \le t_2$. There exists a constant $C = C(d,s,\tau,t_1,t_2) > 0$ such that, for every $1 < \lambda < 2$, $J \ge 1$, $K \ge 1$ and $t \in [t_1,t_2]$, and for every divergence-free $v \in \dot H^\sigma(\TT^d)$ with $\norm{ v }{ \dot H^{s+2} } \le K$, we have
	\begin{equation} \label{eqn:B bd}
		\norm{ B(v) }{ X_J^s(\lambda,t) }^2
		\le C \re^{(\frac\lambda2 + \tau) t J} \bigl( \norm{ v }{ X_J^s(\lambda,t) }^4 + K \norm{ v }{ X_J^s(\lambda,t) }^3 \bigr).
	\end{equation}
\end{lemma}
\begin{proof}
	Since $B(v) = B(v, P_{\le J} v) + B(v, P_{> J} v)$, we bound the two contributions separately. Since $\Pi$ is a Fourier multiplier with norm at most $1$, which commutes with the projectors $P_j$, \eqref{eqn:X prod est low freq} (Lemma~\ref{lem:X prod est}, with $\eta = \frac\tau2$) gives
	\begin{equation*}
		\norm{ B( v , P_{\le J} v ) }{ X_J^s(\lambda,t) }
		\lesssim_\tau \re^{(\frac\lambda4+\frac\tau2) t J} \norm{ v }{ X_J^s(\lambda,t) }^2.
	\end{equation*}
	For the high-frequency contribution, we have
	\begin{equation*}
		\begin{aligned}
			& \norm{ B( v , P_{> J} v ) }{ X_J^s(\lambda,t) }^2 \\
			& \le ( J + 1 ) \re^{ \lambda t J } \Bigl( \sum_{ i \le J } \norm{ B( P_i v , P_{> J} v ) }{ \dot H^s } + \norm{ B( P_{> J} v ) }{ \dot H^s } \Bigr)^2 \\
			& \lesssim J \re^{ \lambda t J } \Bigl( \sum_{ i \le J } \norm{ P_i v }{ \dot H^s } + \norm{ P_{> J} v }{ \dot H^s } \Bigr)^2 \norm{ P_{> J} v }{ \dot H^s } \norm{ P_{> J} v }{ \dot H^{s+2} } \\
			& \lesssim_\tau K \re^{ (\frac\lambda2+\tau) t J } \norm{ v }{ X_J^s(\lambda,t) }^3,
		\end{aligned}
	\end{equation*}
	where for the second inequality we used the algebra property and interpolated between $\dot H^s$ and $\dot H^{s + 2}$ using \eqref{eqn:interpolation}, and for the last one we used $\norm{ P_i v }{ \dot H^s } \le \re^{-\frac\lambda2 t i} \norm{ v }{ X_J^s(\lambda,t) }$, $\norm{ P_{>J} v }{ \dot H^s } \le \re^{-\frac\lambda2 t J} \norm{ v }{ X_J^s(\lambda,t) }$, $\norm{ P_{>J} v }{ \dot H^{s+2} } \le K$, and $J \lesssim_\tau \re^{\tau t J}$ (since $t \ge t_1$). All the implicit constants are uniform for $t \in [t_1,t_2]$, by Lemma~\ref{lem:X prod est}. Squaring the first bound and adding the second one gives \eqref{eqn:B bd}.
\end{proof}
We also record the following a priori bound for the forced heat equation.
\begin{lemma} \label{lem:forced heat a priori bd}
	Let $u \colon [0,T] \times \TT^d \to \RR^d$ be a smooth solution to the forced (vector) heat equation
	\begin{equation} \label{eqn:forced heat}
		\begin{cases}
			& ( \partial_t - \Delta ) u = \nabla \cdot F, \\
			& u(0) = u_0,
		\end{cases}
	\end{equation}
	where $F \colon [0,T] \times \TT^d \to (\RR^d)^{\otimes 2}$ is smooth, $\nabla \cdot F = ( \nabla \cdot F_i : i = 1,\ldots,d )$, and $u_0$ is smooth. Fix $\lambda,t_0,s,J$ defining $X_J^s(\lambda,t_0)$. Then
	\begin{equation}
		\norm{ u }{C_T X_J^s(\lambda,t_0)}^2 + \norm{ \nabla u }{ L^2_T X_J^s(\lambda,t_0) }^2
		\le 2 \bigl( \norm{ u_0 }{X_J^s(\lambda,t_0)}^2 + \norm{ F }{ L^2_T X_J^s(\lambda,t_0) }^2 \bigr).
	\end{equation}
\end{lemma}
\begin{proof}
	Taking, at each level $j$, the $\dot H^s$-inner product of \eqref{eqn:forced heat} with $P_j u$ (respectively $P_{>J} u$), summing with the weights of $X_J^s(\lambda,t_0)$, and integrating by parts, we obtain
	\begin{equation*}
		\begin{aligned}
			& \tfrac12 \left.\frac{\rd}{\rd \tau}\right|_{\tau=t} \Bigl[ \norm{ u(\tau) }{X_J^s(\lambda,t_0)}^2 \Bigr] \\
			& = \sum_{j \le J} \re^{\lambda t_0 j} \innerprod{ P_j u(t) }{ \Delta u(t) }{ \dot H^s } + \re^{\lambda t_0 J} \innerprod{ P_{>J} u(t) }{ \Delta u(t) }{ \dot H^s } \\
			& \hskip 10pt + \sum_{j \le J} \re^{\lambda t_0 j} \innerprod{ P_j u(t) }{ \nabla \cdot F(t) }{ \dot H^s } + \re^{\lambda t_0 J} \innerprod{ P_{>J} u(t) }{ \nabla \cdot F(t) }{ \dot H^s } \\
			& = - \Bigl( \sum_{j \le J} \re^{\lambda t_0 j} \norm{ P_j \nabla u(t) }{ \dot H^s }^2 + \re^{\lambda t_0 J} \norm{ P_{>J} \nabla u(t) }{ \dot H^s }^2 \Bigr) \\
			& \hskip 10pt - \Bigl( \sum_{j \le J} \re^{\lambda t_0 j} \innerprod{ P_j \nabla u(t) }{ F(t) }{ \dot H^s } + \re^{\lambda t_0 J} \innerprod{ P_{>J} \nabla u(t) }{ F(t) }{ \dot H^s } \Bigr) \\
			& \le - \tfrac12 \Bigl( \sum_{j \le J} \re^{\lambda t_0 j} \norm{ P_j \nabla u(t) }{ \dot H^s }^2 + \re^{\lambda t_0 J} \norm{ P_{>J} \nabla u(t) }{ \dot H^s }^2 \Bigr) \\
			& \hskip 10pt + \tfrac12 \Bigl( \sum_{j \le J} \re^{\lambda t_0 j} \norm{ P_j F(t) }{ \dot H^s }^2 + \re^{\lambda t_0 J} \norm{ P_{>J} F(t) }{ \dot H^s }^2 \Bigr),
		\end{aligned}
	\end{equation*}
	where the last step follows from the Cauchy--Schwarz inequality and $ab \le \frac12 a^2 + \frac12 b^2$. Rearranging and integrating the above yields the desired bound.
\end{proof}

\subsection{The Wick--Isserlis theorem} \label{subsec:wick}
We will use the Wick--Isserlis theorem to compute expectations of products of Gaussians.
\begin{thm}[Wick--Isserlis] \label{thm:Wick-Isserlis}
	Let $Y_1,\ldots,Y_N$ be centred, jointly Gaussian (real or complex) random variables. Then
	\begin{equation}
		\EE \bigl[ Y_1 \cdots Y_N \bigr]
		= \sum_{ \mathfrak{p} } \prod_{ \{ a, b\} \in \mathfrak{p} } \EE \bigl[ Y_a Y_b \bigr],
	\end{equation}
	where the sum is over all pairings $\mathfrak{p}$ of $\{1,\ldots,N\}$. In particular, the expectation is zero when $N$ is odd.
\end{thm}

\section{Gaussian estimates} \label{sec:gaussian}
In this section we collect the estimates which hold $\mu_t$-almost surely, and will be exploited to derive the contradiction at the end of Section~\ref{sec:sing} (see Corollary~\ref{lem:soln not too smooth}).

\label{subsec:pf prob}
Fix small numbers $\eta_0,\eta_1,\eta_2 > 0$. Given parameters $\vec M=(M_0,M_1) \in \ZZ^2_{\ge 0}$ and the sequence $(j_k)$ as in Lemma~\ref{lem:Aj count}, let
\begin{align}
		& \cG_{t;\vec M, j}^{(1)}
	\coloneq \bigl\{ u \in \dot H^\sigma(\TT^d) \colon \re^{\lambda t j} \norm{ P_j u }{ \dot H^s }^2 \le M_0^2 |A_j| \re^{(\lambda - 2 + \eta_0) t j}, \notag \\
	\label{eqn:G1 defn} & \hskip 10pt \re^{\lambda t j} \bigl| \lambda j \norm{ P_j u }{ \dot H^s }^2 - 2 \norm{ P_j u }{ \dot H^{s+1} }^2 + 2 \innerprod{ P_j u }{ B(u) }{ \dot H^s } \bigr| \le M_1^2 |A_j| \re^{(\lambda - \frac32 + \eta_1) t j} \bigr\}, \\
	\label{eqn:G2 defn}
	& \cG_{t; j}^{(2)}
	\coloneq \bigl\{ u \in \dot H^\sigma(\TT^d) \colon \re^{\lambda t j} \norm{ P_j B(u) }{ \dot H^s }^2 \ge |A_j| \re^{(\lambda - 1 - \eta_2) t j} \bigr\}, \\
	& \label{eqn:G defn}
	\cG_t
	\coloneq \Bigl( \bigcup_{\vec M} \bigcap_j \cG_{t;\vec M,j}^{(1)} \Bigr) \cap \Bigl( \bigcup_{k_0} \bigcap_{k \ge k_0} \cG_{t; j_k}^{(2)} \Bigr).
\end{align}
Here, by \eqref{eqn:rj'} in Section~\ref{sec:local}, if $u$ is a solution to \eqref{eqn:NSE}, then the quantity bounded in the second condition of \eqref{eqn:G1 defn}, evaluated at $u(t)$, is $\bigl| \frac{\rd}{\rd t} \bigl( \re^{\lambda t j} \norm{ P_j u(t) }{ \dot H^s }^2 \bigr) \bigr|$.  For $t > 0$, the measure $\mu_t$ is concentrated on $C^\infty(\TT^d) \subseteq \dot H^\sigma(\TT^d)$, on which these sets are defined. The sets $\cG_{t;\vec M, j}^{(1)},\cG_{t; j}^{(2)}$ are made up of precisely those $u$ which satisfy the Gaussian predictions on the sizes of the terms involved, at time $t$ and level $j$. More precisely, these predictions hold $\mu_t$-almost surely, in the following sense.
\begin{prop} \label{prop:Gaussian}
	Let $t > 0$ and $1 < \lambda < 2$. Then the set $\cG_t$ defined in \eqref{eqn:G defn} has full $\mu_t$-measure, i.e.
	\begin{equation}
		\mu_t ( \cG_t ) = 1.
	\end{equation}
\end{prop}
\begin{proof}
	We take for granted the following lemma computing various expectations and variances, leaving its proof to the end of the section.
	\begin{lemma} \label{lem:Gaussian expectations}
		Let $t > 0$, $1 < \lambda < 2$, and let $j \ge 1$. Then, with implicit constants depending on $t$, $s$, $\alpha$, $\lambda$ and $d$,
		\begin{align}
			\label{eqn:Erj}
			& \EE_{u \sim \mu_t} \re^{\lambda t j} \norm{ P_j u }{ \dot H^s }^2
			\sim |A_j| \re^{(\lambda - 2)tj} j^{s - \alpha}, \\
			\label{eqn:Er'j}
			& \EE_{u \sim \mu_t} \bigl| \re^{\lambda t j} \innerprod{ P_j u }{ B(u) }{ \dot H^s } \bigr|
			\lesssim |A_j| \re^{(\lambda - \frac32)tj} j^{s - \frac32 \alpha}, \\
			\label{eqn:EZj}
			& \EE_{u \sim \mu_t} \re^{\lambda t j} \norm{ P_j B(u) }{ \dot H^s }^2
			\sim |A_j| \re^{(\lambda - 1) t j} j^{s - 2\alpha - \beta_d}, \\
			\label{eqn:varZj}
			& \var_{u \sim \mu_t} \bigl( \re^{\lambda t j} \norm{ P_j B(u) }{ \dot H^s }^2 \bigr)
			\lesssim |A_j| \re^{2 (\lambda - 1) t j} j^{2s - 4\alpha - 2\beta_d}.
		\end{align}
		Here $\beta_d \coloneq 1$ if $d = 2$ and $\beta_d \coloneq 0$ if $d \ge 3$, and the lower bound in \eqref{eqn:EZj} holds for $j \ge j_0$, for some $j_0 = j_0(d)$.
	\end{lemma}
	By Markov's inequality and \eqref{eqn:Erj}, we have (only the $j$ with $|A_j| \ge 1$ matter, since $P_j = 0$ if $|A_j| = 0$)
	\begin{align*}
		\mu_t \Bigl( \sup_{j \colon |A_j| \ge 1} \frac{ \re^{\lambda t j} \norm{ P_j u }{ \dot H^s }^2 }{ |A_j| \re^{(\lambda - 2 + \eta_0)tj} } = \infty \Bigr)
		& \le \lim_{R \to \infty} \sum_{j \colon |A_j| \ge 1} \mu_t \Bigl( \frac{ \re^{\lambda t j} \norm{ P_j u }{ \dot H^s }^2 }{ |A_j| \re^{(\lambda - 2 + \eta_0)tj} } > R \Bigr) \\
		& \lesssim \lim_{R \to \infty} \frac1{R} \sum_j \re^{-\frac{\eta_0}2 tj}
		= 0.
	\end{align*}
	Likewise (using \eqref{eqn:Erj} and \eqref{eqn:Er'j})
	\begin{equation*}
		\mu_t \Bigl( \sup_{j \colon |A_j| \ge 1} \frac{ \re^{\lambda t j} \bigl| \lambda j \norm{ P_j u }{ \dot H^s }^2 - 2 \norm{ P_j u }{ \dot H^{s+1} }^2 + 2 \innerprod{ P_j u }{ B(u) }{ \dot H^s } \bigr| }{ |A_j| \re^{(\lambda - \frac32 + \eta_1)tj}} = \infty \Bigr) = 0.
	\end{equation*}
	It follows that
	\begin{equation*}
		\mu_t \Bigl( \bigcup_{\vec M} \bigcap_j \cG_{t;\vec M,j}^{(1)} \Bigr) = 1.
	\end{equation*}
	On the other hand, by \eqref{eqn:EZj}, there is $j_\ast = j_\ast(t)$ such that $\frac12\EE_{v\sim\mu_t} \re^{\lambda t j} \norm{ P_j B(v) }{ \dot H^s }^2 \ge |A_j| \re^{(\lambda - 1 - \eta_2) t j}$ for every $j \ge j_\ast$ with $|A_j| \ge 1$ (we may assume that $j_\ast \ge j_0$), since the polynomial factor $j^{s - 2\alpha - \beta_d}$ is beaten by $\re^{\eta_2 t j}$. By Chebyshev's inequality, \eqref{eqn:EZj} and \eqref{eqn:varZj}, for every $k$ such that $j_k \ge j_\ast$, we have
	\begin{align*}
		& \mu_t \bigl( \bigl( \cG_{t;j_k}^{(2)} \bigr)^c \bigr) \\
		& \le \mu_t \bigl( \re^{\lambda t j_k} \norm{ P_{j_k} B(u) }{ \dot H^s }^2 < \tfrac12 \EE_{v \sim \mu_t} \re^{\lambda t j_k} \norm{ P_{j_k} B(v) }{ \dot H^s }^2 \bigr) \\
		& \le \mu_t \Bigl( \bigl| \re^{\lambda t j_k} \norm{ P_{j_k} B(u) }{ \dot H^s }^2 - \EE_{v \sim \mu_t} \re^{\lambda t j_k} \norm{ P_{j_k} B(v) }{ \dot H^s }^2 \bigr| \\
		& \hskip 60pt > \tfrac12 \EE_{v \sim \mu_t} \re^{\lambda t j_k} \norm{ P_{j_k} B(v) }{ \dot H^s }^2 \Bigr) \\
		& \le \frac{ 4 \var_{v \sim \mu_t} \bigl( \re^{\lambda t j_k} \norm{ P_{j_k} B(v) }{ \dot H^s }^2 \bigr) }{ \bigl( \EE_{v \sim \mu_t} \re^{\lambda t j_k} \norm{ P_{j_k} B(v) }{ \dot H^s }^2 \bigr)^2 }
		\lesssim |A_{j_k}|^{-1}.
	\end{align*}
	By Lemma~\ref{lem:Aj count}, the right-hand side is summable, and so by Borel--Cantelli, we obtain
	\begin{equation*}
		\mu_t \Bigl( \bigcup_{k_0} \bigcap_{k \ge k_0} \cG_{t; j_k}^{(2)} \Bigr) = 1,
	\end{equation*}
	which completes the proof.
\end{proof}
We end this section with a proof of Lemma~\ref{lem:Gaussian expectations}.
The proof is a direct computation of first and second moments. The key observation is that, if $n_1 + n_2 = n$ and we write $n_1 = \frac n2 - m$, $n_2 = \frac n2 + m$, then $|n_1|^2 + |n_2|^2 = \frac12 |n|^2 + 2 |m|^2$. Therefore the products of Gaussian Fourier coefficients appearing in $\widehat{B(u)}(n)$ have size $\re^{-\frac t2 |n|^2}$, and they are concentrated on pairs with $m = O(1)$, i.e.\ on pairs of frequencies which are almost parallel to $n$.  Indeed, $|n_1|^2 + |n_2|^2$ is minimised at $m = 0$, i.e.\ at $n_1 = n_2 = \frac n2$, so that the exponential weight $\re^{-2t(|n_1|^2 + |n_2|^2)} = \re^{-t|n|^2 - 4t|m|^2}$ concentrates the sums around $|m| = O(1)$.

 The Wick--Isserlis theorem reduces all expectations to sums over such pairs. The only delicate point is a lower bound for the resulting geometric factor, which in $d = 2$ exhibits an additional cancellation and is therefore treated separately at the end of the proof.
\begin{proof}[Proof of Lemma~\ref{lem:Gaussian expectations}]
	To prove \eqref{eqn:Erj}, compute directly that
	\begin{equation*}
		\EE_{u \sim \mu_t} \re^{\lambda t j} \norm{ P_j u }{ \dot H^s }^2
		= \re^{\lambda t j} \EE \Bigl[ \sum_{|n|^2 = j} |n|^{2s} \re^{-2t|n|^2} \frac{ |g_n|^2 }{ |n|^{2\alpha} } \Bigr]
		= (d - 1) |A_j| \re^{(\lambda - 2) t j} j^{s - \alpha},
	\end{equation*}
	where the expectation was calculated using
	\begin{equation*}
		\EE |g_n|^2
		= \EE \tr\bigl( (\Pi_n G_n)^\rt \Pi_n \bar{G_n} \bigr)
		= \tr \Pi_n = d - 1
	\end{equation*}
	after cycling the trace. The left-hand side of \eqref{eqn:Er'j}
	can be estimated easily, since by the Cauchy--Schwarz inequality, we have
	\begin{equation*}
		\begin{aligned}
			& \EE_{u \sim \mu_t} | \re^{\lambda t j} \innerprod{ P_j u }{ B(u) }{ \dot H^s } | \\
			& \le \bigl( \EE_{u \sim \mu_t} \re^{\lambda t j} \norm{ P_j u }{ \dot H^s }^2 \bigr)^{\frac12} \bigl( \EE_{u \sim \mu_t} \re^{\lambda t j} \norm{ P_j B(u) }{ \dot H^s }^2 \bigr)^{\frac12} \\
			& \lesssim |A_j| \re^{ (\lambda- \frac32) tj } j^{s - \frac{3\alpha}2},
		\end{aligned}
	\end{equation*}
	where we used \eqref{eqn:Erj} and the upper bound in \eqref{eqn:EZj} for the second inequality, which is proven next (in $d = 2$, the upper bound in \eqref{eqn:EZj} is even better). Since $|n|^2 = j$ on $A_j$, using incompressibility and symmetrising, we have
	\begin{equation} \label{eqn:EZj simplified1} \begin{aligned}
		& \EE_{u \sim \mu_t} \re^{\lambda t j} \norm{ P_j B(u) }{ \dot H^s }^2 \\
		& = \re^{\lambda t j} j^s \sum_{|n|^2 = j} \EE \Bigl| \sum_{n_1 + n_2 = n} \frac{ \re^{ - t ( |n_1|^2 + |n_2|^2 ) } }{ 2 |n_1|^\alpha |n_2|^\alpha }  \Pi_n \bigl[ (g_{n_1} \cdot n) g_{n_2} + (g_{n_2} \cdot n) g_{n_1} \bigr] \Bigr|^2.
	\end{aligned} \end{equation}
	To prove \eqref{eqn:EZj}, we expand the square above and handle all the various pairings in turn. We claim that by Theorem~\ref{thm:Wick-Isserlis} the expectation
	\begin{equation}
		\EE \bigl[ \Pi_n \bigl[ (g_{n_1} \cdot n) g_{n_2} + (g_{n_2} \cdot n) g_{n_1} \bigr] \cdot \bar{ \Pi_n \bigl[ (g_{n_3} \cdot n) g_{n_4} + (g_{n_4} \cdot n) g_{n_3} \bigr] } \bigr]
	\end{equation}
	vanishes unless $\{n_1,n_2\} = \{n_3,n_4\}$.  Indeed, by Theorem~\ref{thm:Wick-Isserlis} applied to the components of $g_{n_1}$, $g_{n_2}$, $\bar g_{n_3}$, $\bar g_{n_4}$, this expectation is a sum of products of covariances. Since $\EE[ g_a \otimes g_b ] = 0$ unless $a = -b$, $\EE[ g_a \otimes \bar g_b ] = 0$ unless $a = b$, and $n_1 = -n_2$ would force $n = 0$, every non-vanishing pairing matches $\{n_1,n_2\}$ with $\{n_3,n_4\}$. Moreover, the term with $n_1 = n_2 = \frac n2$ vanishes, since $g_{\frac n2} \cdot n = 0$. By symmetry, it follows that
	\begin{equation} \label{eqn:EZj simplified2} \begin{aligned}
		& \EE_{u \sim \mu_t} \re^{\lambda t j} \norm{ P_j B(u) }{ \dot H^s }^2 \\
		& = \tfrac12 \re^{\lambda t j} j^s \sum_{|n|^2 = j} \sum_{n_1 + n_2 = n} \frac{ \re^{ - 2 t ( |n_1|^2 + |n_2|^2 ) } }{ |n_1|^{2\alpha} |n_2|^{2\alpha} } \EE \bigl| \Pi_n \bigl[ (g_{n_1} \cdot n) g_{n_2} + (g_{n_2} \cdot n) g_{n_1} \bigr] \bigr|^2.
	\end{aligned} \end{equation}
	The remaining pairings give
	\begin{equation} \label{eqn:Zj pairings} \begin{aligned}
		& \EE \bigl| \Pi_n \bigl[ (g_{n_1} \cdot n) g_{n_2} + (g_{n_2} \cdot n) g_{n_1} \bigr] \bigr|^2 \\
		& = | \Pi_{n_1} n |^2 \tr( \Pi_n \Pi_{n_2} ) + | \Pi_{n_2} n |^2 \tr( \Pi_n \Pi_{n_1} ) + 2 ( \Pi_n \Pi_{n_1} n ) \cdot ( \Pi_n \Pi_{n_2} n ).
	\end{aligned} \end{equation}
	We now bound the summand in \eqref{eqn:EZj simplified2} from above and from below. We parametrise the sum over $n_1 + n_2 = n$ as $n_1 = \frac n2 - m$ and $n_2 = \frac n2 + m$, with $m \in \ZZ^d - \frac n2$, so that $|n_1|^2 + |n_2|^2 = \frac12 |n|^2 + 2|m|^2$. Since $\Pi_{n_1} n = 2 \Pi_{n_1} m$ and $\Pi_{n_2} n = - 2 \Pi_{n_2} m$, the right-hand side of \eqref{eqn:Zj pairings} is $\lesssim |m|^2$. If $|m| \le \frac14 |n|$, then $|n_1|,|n_2| \sim |n|$, so that the weight in \eqref{eqn:EZj simplified2} is $\lesssim \re^{-tj - 4t|m|^2} j^{-2\alpha}$. If $|m| > \frac14 |n|$, then $|n_1|,|n_2| \ge 1$, and the weight is at most $\re^{-tj - 4t|m|^2} \le \re^{-tj - 2t|m|^2} \re^{- \frac t8 |n|^2} \lesssim \re^{-tj - 2t|m|^2} j^{-2\alpha}$. In either case,
	\begin{equation} \label{eqn:Egn1gn2 bd} \begin{aligned}
		& \sum_{n_1 + n_2 = n} \frac{ \re^{ - 2 t ( |n_1|^2 + |n_2|^2 ) } }{ |n_1|^{2\alpha} |n_2|^{2\alpha} } \EE \bigl| \Pi_n \bigl[ (g_{n_1} \cdot n) g_{n_2} + (g_{n_2} \cdot n) g_{n_1} \bigr] \bigr|^2 \\
		& \lesssim \re^{-t j} j^{-2\alpha} \sum_{m \in \ZZ^d - \frac n2} |m|^2 \re^{-2t|m|^2}
		\lesssim \re^{-tj} j^{-2\alpha}.
	\end{aligned} \end{equation}
	Inserting this bound into \eqref{eqn:EZj simplified2} gives the upper bound in \eqref{eqn:EZj}.
	To prove \eqref{eqn:EZj}, it remains to bound \eqref{eqn:EZj simplified2} from below when $d \ge 3$.  Let $C \coloneq \sqrt d / \sin \frac\pi{12}$ and $c \coloneq \frac14$, and let $|n| \gg C$. We claim that, for every $\theta \in \SS^{d-1}$, there exists $m \in \ZZ^d - \frac n2$ such that
	\begin{equation} \label{eqn:good m}
		1 \le |m| \le C
		\qquad \text{and} \qquad
		\tfrac14 \le \sin^2 \angle(\theta,m) \le \tfrac34;
	\end{equation}
	in particular, $|m|^2 \sin^2 \angle(\theta,m) \ge c$ and $|m|^2 \cos^2 \angle(\theta,m) \ge c$. Indeed, let $x \in \SS^{d-1}$ be such that $\angle(\theta,x) = \frac\pi4$, and let $\rho \coloneq \frac12 C$. If $y \in \RR^d$ satisfies $|y - \rho x| \le \rho \sin \frac\pi{12}$, then $\angle(x,y) \le \frac\pi{12}$, so that $\frac\pi6 \le \angle(\theta,y) \le \frac\pi3$, and $\rho (1 - \sin \frac\pi{12}) \le |y| \le \rho (1 + \sin \frac\pi{12})$. Since $\rho \sin \frac\pi{12} = \frac{\sqrt d}2$, and every closed ball of radius $\frac{\sqrt d}2$ contains a point of the translate $\ZZ^d - \frac n2$ of $\ZZ^d$, we can choose such a $y = m \in \ZZ^d - \frac n2$. Then \eqref{eqn:good m} holds, since $\frac14 \le \sin^2 \le \frac34$ on $[\frac\pi6, \frac\pi3]$, and $1 \le \frac\rho2 \le |m| \le 2\rho = C$. Given $n$, let $m$ be as in \eqref{eqn:good m} with $\theta = n / |n|$, and let $n_1 = \frac n2 - m$, $n_2 = \frac n2 + m$. Using \eqref{eqn:Zj pairings}, we have
	\begin{equation} \begin{aligned}
		& \EE \bigl| \Pi_n \bigl[ (g_{n_1} \cdot n) g_{n_2} + (g_{n_2} \cdot n) g_{n_1} \bigr] \bigr|^2 \\
		& \ge | \Pi_{n_1} n |^2 (d - 2 + \cos^2 \angle(n,n_2) ) + | \Pi_{n_2} n |^2 (d - 2 + \cos^2 \angle(n,n_1) ) \\
		& \hskip 15pt - 2 | \Pi_{n_1} n | | \Pi_{n_2} n | \\
		& \ge (1 + \tfrac12) ( | \Pi_{n_1} n |^2 + | \Pi_{n_2} n |^2 ) - 2 | \Pi_{n_1} n | | \Pi_{n_2} n | \\
		& \ge \tfrac12 | \Pi_{n_1} n |^2
		= 2 |m|^2 \sin^2 \angle(n_1,m) \ge |m|^2 \sin^2 \angle(n,m)
		\ge c,
	\end{aligned} \end{equation}
	where the fourth line follows from the fact that $n$, $n_1$, and $n_2$ point in (nearly) the same direction when $|m| \le C$ and $|n| \gg C$ (so that $\cos^2 \angle(n,n_i) \ge \frac12$, and $d - 2 + \cos^2 \angle(n,n_i) \ge \frac32$ since $d \ge 3$), and the second-to-last inequality holds for $|n|$ large enough, depending on $C$ and $c$. Keeping only the term defined by $\pm m$ in \eqref{eqn:EZj simplified2} gives the lower bound required on $\EE_{u\sim\mu_t} \re^{\lambda t j} \norm{ P_j B(u) }{ \dot H^s }^2$, and proves \eqref{eqn:EZj} for $d \ge 3$ and $j \ge j_0$.  This argument uses $d - 2 \ge 1$, and does not apply when $d = 2$; this case is treated at the end of the proof. \\
	Finally, we prove \eqref{eqn:varZj} for $d \ge 3$. Analogously to \eqref{eqn:EZj simplified1}, we have
	\begin{equation} \label{eqn:varZj simplified} \begin{aligned}
		& \var_{u \sim \mu_t} \bigl( \re^{\lambda t j} \norm{ P_j B(u) }{ \dot H^s }^2 \bigr) \\
		& = \tfrac1{16} \re^{2 \lambda t j} j^{2s} \sum_{|n|^2 = j} \sum_{|n'|^2 = j} \sum_{ \substack{ n_1 + n_2 = n \\ n_3 + n_4 = n } } \sum_{ \substack{ n_1' + n_2' = n' \\ n_3' + n_4' = n' } } \\
		& \hskip30pt \frac{ \re^{ - t ( |n_1|^2 + \cdots + |n_4|^2 ) } }{ ( |n_1| \cdots |n_4| )^\alpha } \frac{ \re^{ - t ( |n_1'|^2 + \cdots + |n_4'|^2 ) } }{ ( |n_1'| \cdots |n_4'| )^\alpha } \\
		& \hskip30pt \cdot \cov \bigl( \Pi_n \bigl[ (g_{n_1} \cdot n) g_{n_2} + (g_{n_2} \cdot n) g_{n_1} \bigr] \cdot \bar{ \Pi_n \bigl[ (g_{n_3} \cdot n) g_{n_4} + (g_{n_4} \cdot n) g_{n_3} \bigr] } , \\
		& \hskip40pt \bar{ \Pi_{n'} \bigl[ (g_{n_1'} \cdot n') g_{n_2'} + (g_{n_2'} \cdot n') g_{n_1'} \bigr] } \cdot \Pi_{n'} \bigl[ (g_{n_3'} \cdot n') g_{n_4'} + (g_{n_4'} \cdot n') g_{n_3'} \bigr] \bigr),
	\end{aligned} \end{equation}
	where $\cov(X,Y) = \EE[X\bar Y] - \EE X \bar{ \EE Y }$. By hypercontractivity (see e.g.\ \cite{Bog98}),
	\begin{equation}
		\EE \bigl| \Pi_n \bigl[ (g_{n_1} \cdot n) g_{n_2} + (g_{n_2} \cdot n) g_{n_1} \bigr] \bigr|^4
		\lesssim \bigl( \EE \bigl| \Pi_n \bigl[ (g_{n_1} \cdot n) g_{n_2} + (g_{n_2} \cdot n) g_{n_1} \bigr] \bigr|^2 \bigr)^2
	\end{equation}
	(and similarly for the other indices) and so, by two applications of H\"older's inequality, each covariance in \eqref{eqn:varZj simplified} is
	\begin{equation} \begin{aligned}
		& \lesssim \bigl( \EE \bigl| \Pi_n \bigl[ (g_{n_1} \cdot n) g_{n_2} + (g_{n_2} \cdot n) g_{n_1} \bigr] \bigr|^2 \EE \bigl| \Pi_n \bigl[ (g_{n_3} \cdot n) g_{n_4} + (g_{n_4} \cdot n) g_{n_3} \bigr] \bigr|^2 \bigr)^{\frac12} \\
		& \hskip15pt \bigl( \EE \bigl| \Pi_{n'} \bigl[ (g_{n_1'} \cdot n') g_{n_2'} + (g_{n_2'} \cdot n') g_{n_1'} \bigr] \bigr|^2 \\
		& \hskip30pt \EE \bigl| \Pi_{n'} \bigl[ (g_{n_3'} \cdot n') g_{n_4'} + (g_{n_4'} \cdot n') g_{n_3'} \bigr] \bigr|^2 \bigr)^{\frac12}.
	\end{aligned} \end{equation}
	Moreover, the sums in \eqref{eqn:varZj simplified} keep only pairings with some $n_i = \pm n'_{i'}$ since otherwise the two factors are functions of independent Gaussians. In particular, writing $n_1,n_2 = \frac n2 \mp m_1$ etc., we have $\rho(n,n') \coloneq \min\{ |n - n'| , |n + n'| \} \le 2 ( |m_1| + \cdots + |m_4| )$, and then $|m_1|^2 + \cdots + |m_4|^2 \ge \frac1{16} \rho(n,n')^2$. By \eqref{eqn:Egn1gn2 bd} and its proof,
	\begin{equation} \label{eqn:varZj bd} \begin{aligned}
		& \var_{u \sim \mu_t} \bigl( \re^{\lambda t j} \norm{ P_j B(u) }{ \dot H^s }^2 \bigr) \\
		& \lesssim \re^{ 2 (\lambda - 1) t j } j^{2s - 4\alpha} \sum_{|n|^2,|n'|^2 = j} \sum_{m_1,\ldots,m_4} |m_1| \cdots |m_4| \, \re^{ - 2t ( |m_1|^2 + \cdots + |m_4|^2 ) } \\
		& \hskip 200pt \times \1_{ \{ \rho(n,n') \le 2 ( |m_1| + \cdots + |m_4| ) \} } \\
		& \lesssim \re^{ 2 (\lambda - 1) t j } j^{2s - 4\alpha} \sum_{|n|^2,|n'|^2 = j} \re^{ - \frac t{16} \rho(n,n')^2 }.
	\end{aligned} \end{equation}
	Since $\sum_{n'} \re^{ - \frac t{16} \rho(n,n')^2 } \le 2 \sum_{v \in \ZZ^d} \re^{ - \frac t{16} |v|^2 } \lesssim 1$, we arrive precisely at \eqref{eqn:varZj} for $d \ge 3$.
	\emph{The case $d = 2$.} The identities \eqref{eqn:EZj simplified1}, \eqref{eqn:EZj simplified2} and \eqref{eqn:varZj simplified} hold in every dimension $d \ge 2$, and the argument above gives the upper bounds in \eqref{eqn:EZj} and \eqref{eqn:varZj} with $\beta_d = 0$. In $d = 2$, however, the argument for the lower bound fails, since $d - 2 + \cos^2 \angle(n,n_2) < 1$, and there is an additional cancellation, which improves both bounds by a power of $j$.
	In $d = 2$, we can write $g_{n_i} = \gamma_{n_i} n_i^\perp / |n_i|$, where $\gamma_{n_i} \coloneq g_{n_i} \cdot n_i^\perp / |n_i|$ is a standard complex Gaussian, and $\gamma_{n_1}$, $\gamma_{n_2}$ are independent when $n_1 \ne \pm n_2$. With $n_1 = \frac n2 - m$ and $n_2 = \frac n2 + m$, we have $n_1^\perp \cdot n = n^\perp \cdot m = - n_2^\perp \cdot n$ (recall that $a^\perp \cdot b = - a \cdot b^\perp$), $n_2^\perp - n_1^\perp = 2 m^\perp$, and $\Pi_n m^\perp = \frac{n \cdot m}{|n|^2} n^\perp$. Hence
	\begin{equation} \label{eqn:Zj pairings 2d}
		\Pi_n \bigl[ (g_{n_1} \cdot n) g_{n_2} + (g_{n_2} \cdot n) g_{n_1} \bigr]
		= \frac{ 2 \gamma_{n_1} \gamma_{n_2} }{ |n_1| |n_2| } \frac{ (n^\perp \cdot m) (n \cdot m) }{ |n|^2 } n^\perp,
	\end{equation}
	and, since $\EE |\gamma_{n_1} \gamma_{n_2}|^2 = 1$ for $m \ne 0$,
	\begin{equation} \label{eqn:Zj pairings 2d moment}
		\begin{aligned}
			\EE \bigl| \Pi_n \bigl[ (g_{n_1} \cdot n) g_{n_2} + (g_{n_2} \cdot n) g_{n_1} \bigr] \bigr|^2
			& = \frac{ 4 (n^\perp \cdot m)^2 (n \cdot m)^2 }{ |n|^2 |n_1|^2 |n_2|^2 } \\
			& = \frac{ 4 |n|^2 |m|^4 \sin^2 \angle(n,m) \cos^2 \angle(n,m) }{ |n_1|^2 |n_2|^2 }.
		\end{aligned}
	\end{equation}
	If $|m| \le \frac14 |n|$, then $|n_1|, |n_2| \ge \frac14 |n|$, and \eqref{eqn:Zj pairings 2d moment} is $\lesssim |m|^4 |n|^{-2} \lesssim |m|^4 j^{-1}$. If $|m| > \frac14 |n|$, then \eqref{eqn:Zj pairings 2d moment} is at most $4 |n|^2 |m|^4$, while the weight in \eqref{eqn:EZj simplified2} is at most $\re^{-tj - 4t|m|^2} \le \re^{-tj - 2t|m|^2} \re^{-\frac t8 |n|^2}$, and the last factor absorbs any power of $|n|$. Arguing as in \eqref{eqn:Egn1gn2 bd}, we obtain
	\begin{equation} \label{eqn:Egn1gn2 bd 2d}
		\begin{aligned}
			& \sum_{n_1 + n_2 = n} \frac{ \re^{ - 2 t ( |n_1|^2 + |n_2|^2 ) } }{ |n_1|^{2\alpha} |n_2|^{2\alpha} } \EE \bigl| \Pi_n \bigl[ (g_{n_1} \cdot n) g_{n_2} + (g_{n_2} \cdot n) g_{n_1} \bigr] \bigr|^2 \\
			& \lesssim \re^{ - t j } j^{ - 2 \alpha - 1 } \sum_{m \in \ZZ^2 - \frac n2} (1 + |m|^4) \re^{ - 2t |m|^2}
			\lesssim \re^{ - t j } j^{ - 2 \alpha - 1 },
		\end{aligned}
	\end{equation}
	and inserting this into \eqref{eqn:EZj simplified2} gives the upper bound in \eqref{eqn:EZj} with $\beta_2 = 1$. For the variance, we argue exactly as in the case $d \ge 3$, the only difference being that, when $|m_i| \le \frac14 |n|$, each of the four second moments appearing in the bound for the covariances is now $\lesssim |m_i|^4 j^{-1}$ instead of $\lesssim |m_i|^2$. Therefore each covariance carries an additional factor $j^{-2}$, so that \eqref{eqn:varZj bd} holds with $j^{2s - 4\alpha}$ replaced by $j^{2s - 4\alpha - 2}$ (and $|m_1| \cdots |m_4|$ replaced by $|m_1|^2 \cdots |m_4|^2$, which does not affect the summation), while the terms with some $|m_i| > \frac14|n|$ are again controlled by the Gaussian weights. This gives \eqref{eqn:varZj} with $\beta_2 = 1$.

	 We now prove the lower bound. As in the case $d \ge 3$, we keep only the terms with $|m| \le C$, but now the second moment \eqref{eqn:Zj pairings 2d moment} carries an additional factor $j^{-1}$. By \eqref{eqn:good m} with $\theta = n / |n|$, there exists $m \in \ZZ^2 - \frac n2$ with $1 \le |m| \le C$ and $|m|^4 \sin^2 \angle(n,m) \cos^2 \angle(n,m) \ge \frac1{16}$. Let now $|n| \ge 4C$, so that $|m| \le \frac14 |n|$ and $\frac14 |n| \le |n_1|, |n_2| \le \frac34 |n|$. Then \eqref{eqn:Zj pairings 2d moment} gives
	\begin{equation*}
		\EE \bigl| \Pi_n \bigl[ (g_{n_1} \cdot n) g_{n_2} + (g_{n_2} \cdot n) g_{n_1} \bigr] \bigr|^2
		\ge \frac{ |n|^2 }{ 4 ( \frac34 |n| )^4 }
		\gtrsim j^{-1},
	\end{equation*}
	while the corresponding weight in \eqref{eqn:EZj simplified2} is $\re^{-tj - 4t|m|^2} |n_1|^{-2\alpha} |n_2|^{-2\alpha} \gtrsim \re^{-tj} j^{-2\alpha}$. Keeping only the terms defined by $\pm m$ in \eqref{eqn:EZj simplified2}, we obtain $\EE_{u \sim \mu_t} \re^{\lambda t j} \norm{ P_j B(u) }{ \dot H^s }^2 \gtrsim |A_j| \re^{(\lambda - 1) t j} j^{s - 2\alpha - 1}$ for $j \ge 16 C^2$, which is the lower bound in \eqref{eqn:EZj} with $\beta_2 = 1$. Finally, since the ratio between \eqref{eqn:varZj} and the square of \eqref{eqn:EZj} is $\lesssim |A_j|^{-1}$ in every dimension, the proof of Proposition~\ref{prop:Gaussian} applies verbatim in $d = 2$.

\end{proof}

\section{Local theory} \label{sec:local}
In this section we collect the properties of the flow of \eqref{eqn:NSE} that are needed in the proof of Theorem~\ref{thm:main precise}. Since \eqref{eqn:NSE} is autonomous, it is convenient to consider solutions on general time intervals: for $0 \le T_0 < T$, we say that $u \colon [T_0,T] \times \TT^d \to \RR^d$ is a solution to \eqref{eqn:NSE} on $[T_0,T]$ if $u(T_0 + \cdot)$ is a solution to \eqref{eqn:NSE} on $[0, T - T_0]$ in the sense of Definition~\ref{defn:solution}.

The relevant properties of the flow are summarised in the following proposition.
\begin{prop}[Local theory] \label{prop:local theory}
	Let $\sigma$ be as in Definition~\ref{defn:solution}, let $\frac d2 < s \le \sigma - 6$ and $J \ge 1$, and let $u$ be a solution to \eqref{eqn:NSE} on $[T_0,T]$, for some $0 \le T_0 < T$, with $u(T_0) = u_0$. Then the following hold.
	\begin{enumerate}[label=(\roman*)]
		\item \emph{(Semi-linear time of existence in $X_J^s(\lambda,t_0)$)} Let $1 < \lambda < 2$ and $t_0 > 0$. For every $\tau > 0$ there exists a constant $c > 0$, depending only on $d$, $s$, $t_0$ and $\tau$, and uniform for $t_0$ in compact subsets of $(0,\infty)$, such that, setting
		\begin{equation} \label{eqn:T ast}
			T_\ast(M, J) = c (1 + M )^{-2} \re^{ - (\frac\lambda4 + \tau) t_0 J }
		\end{equation}
		for $M \ge 0$, $J \ge 1$,
		for
		\[T_0 \le t \le T_0 + \min\{T_\ast(\norm{ u_0 }{ X_J^s(\lambda,t_0) }, J), T - T_0\}\]
		we have that
		\begin{align*}
			\norm{ u(t) }{ X_J^s(\lambda,t_0) } &\le 2 \norm{ u_0 }{ X_J^s(\lambda,t_0) }, \\
			\norm{ u(t) - S(t - T_0) u_0 }{ X_J^s(\lambda,t_0) }
			&\lesssim_\tau \norm{ u_0 }{ X_J^s(\lambda,t_0) }^2.
		\end{align*}
		\item \emph{(Second derivative bound)}
		There exist $1 < \lambda_0 < 2$ and $\delta_0 > 0$ such that the following holds for every $\lambda_0 < \lambda < 2$ and every $0 < \delta < \delta_0$. Let $K \ge 1$, $j \le J$ and $t_1 \in [T_0, T)$ with $t_1 > 0$, and suppose that $\norm{ u }{ C([T_0,T];\dot H^\sigma)} \le K$. Then there exist $0 < \eta = \eta(\lambda,\delta) \ll 1$, $A = A(\lambda,\delta) > 0$ and a constant $C = C(d,s,\lambda,\delta,t_1,T) > 0$, independent of $j$, $J$ and $K$, such that, setting $\lambda_J \coloneq (\lambda - 1 - \eta)J$,
		\begin{equation} \label{eqn:convexity}
			\begin{aligned}
				\frac{\rd^2}{\rd t^2} \bigl( \re^{\lambda t j} \norm{ P_j u(t) }{ \dot H^s }^2 \bigr)
				& \ge \re^{\lambda t j} \norm{ P_j B(u(t)) }{ \dot H^s }^2 - j^2 \re^{\lambda t j} \norm{ P_j u(t) }{ \dot H^s }^2 \\
				& \hskip 15pt - \re^{ - \delta tJ} \norm{ B(u(t)) }{ X_J^s(\lambda,t) }^2 - \norm{ u(t) }{ X_J^s(\lambda,t) }^{A} - C K^4 \re^{\lambda_J t}
			\end{aligned}
		\end{equation}
		for all $t \in [t_1,T]$.
	\end{enumerate}
\end{prop}
The rest of this section is devoted to the proof of Proposition~\ref{prop:local theory}. Item (i) follows from Proposition~\ref{prop:continuity}, after replacing $0$ with $T_0$, and item (ii) is Proposition~\ref{lem:r'' bd}.

Let us comment on the two items. Item (i) is a standard semi-linear local well-posedness statement in the weighted spaces $X_J^s(\lambda,t_0)$. These spaces measure the decay $\norm{ P_j f }{ \dot H^s } \lesssim \re^{-\frac\lambda2 t_0 j}$, $j \le J$, which is satisfied by typical samples of $\mu_{t_0}$ since $\lambda < 2$ (see Section~\ref{sec:gaussian}). The crucial point is the length $\re^{-(\frac\lambda4 + \tau) t_0 J}$ of the time interval. Since $\frac\lambda4 < \frac12$, it is much longer than the time scale $\re^{-\frac12 t_0 J}$ on which the nonlinearity produces a detectable change in the frequencies $|n|^2 = J$, if one were to assume $u(t) \in \cG_t$ uniformly. Item (ii) is the quantitative convexity estimate that allows us to extend this argument outside of the times $t$ such that $u(t) \in \cG_t$. More precisely, its right-hand side contains the nonnegative term $\re^{\lambda t j} \norm{ P_j B(u(t)) }{ \dot H^s }^2$, which is large when $u(t)$ looks like a typical sample of $\mu_t$, and is bounded below by $0$ otherwise. The error terms will be controlled in Section~\ref{sec:sing} by means of item (i) and of Lemma~\ref{lem:Z>E}.

We now prove Proposition~\ref{prop:local theory}\,(i). The idea is to keep track of the first time at which \,(i) fails, and use the estimates leading to a standard Banach fixed point argument to carefully keep track of the existence time. Note that we do not actually need to perform the fixed point to show existence of the solution, in view of the assumption that $u$ is a solution on $[0,T]$ in Proposition~\ref{prop:continuity} (recall from \eqref{eqn:WP defn} that our results are conditional on the existence of mild solutions).
\begin{prop} \label{prop:continuity}
	Let $u : [0,T] \times \TT^d \to \RR^d$ be a solution to \eqref{eqn:NSE} with initial datum $u_0$. Fix some $\lambda,t_0,s,J$ defining $\norm{ \,\cdot\, }{X_J^s(\lambda,t_0)}$, with $t_0 > 0$. Then, for every $\tau > 0$, there exists a constant $c > 0$, depending only on $d$, $s$, $t_0$ and $\tau$, and uniform for $t_0$ in compact subsets of $(0,\infty)$, such that, with $T_\ast(M, J) = c (1 + M )^{-2} \re^{ - (\frac\lambda4 + \tau) t_0 J }$ for $M \ge 0$ and $J \ge 1$, it holds that
	\begin{align*}
		\norm{ u(t) }{ X_J^s(\lambda,t_0) } &\le 2 \norm{ u_0 }{ X_J^s(\lambda,t_0) }, \\
		\norm{ u(t) - S(t) u_0 }{ X_J^s(\lambda,t_0) }
		&\lesssim_\tau \norm{ u_0 }{ X_J^s(\lambda,t_0) }^2
	\end{align*}
	for $0 \le t \le \min\{T_\ast(\norm{ u_0 }{ X_J^s(\lambda,t_0) }, J), T\}$.
\end{prop}
\begin{proof}
	Throughout the proof, we write $T_\ast$ for $T_\ast(\norm{ u_0 }{ X_J^s(\lambda,t_0) }, J)$. By the Duhamel formula, we have
	\begin{equation} \label{eqn:Duhamel}
		\begin{aligned}
			u(t) - S(t) u_0
			& = \int_0^t S(t - t') B( u(t') , P_{\le J} u(t') ) \, \rd t' \\
			& \hskip 15pt + \int_0^t S(t - t') B( u(t') , P_{> J} u(t') ) \, \rd t' \\
			& \eqcolon \mathrm{I} + \mathrm{II}.
		\end{aligned}
	\end{equation}
	Using the bilinear estimate \eqref{eqn:X prod est low freq} (Lemma~\ref{lem:X prod est}), for every $\tau > 0$ we have
	\begin{equation} \label{eqn:I bd}
		\norm{ \mathrm{I} }{ C_T X_J^s(\lambda,t_0) }
		\lesssim T \norm{ (u \cdot \nabla) P_{\le J} u }{ C_T X_J^s(\lambda,t_0) }
		\lesssim_\tau \re^{ (\frac\lambda4 + \tau) t_0 J } T \norm{ u }{ C_T X_J^s(\lambda,t_0) }^2.
	\end{equation}
	Meanwhile, the integral $\mathrm{II}$ solves the equation
	\begin{equation*}
		\begin{cases}
			& ( \partial_t - \Delta ) \mathrm{II} = - \Pi \nabla \cdot ( u \otimes P_{>J} u ), \\
			& \mathrm{II}(0) = 0,
		\end{cases}
	\end{equation*}
	and so, by Lemma~\ref{lem:forced heat a priori bd},
	\begin{equation} \label{eqn:II bd}
		\norm{ \mathrm{II} }{ C_T X_J^s(\lambda,t_0)}
		\lesssim \norm{ u \otimes P_{>J} u }{ L^2_T X_J^s(\lambda,t_0) }
		\lesssim \sqrt T \norm{ u }{ C_T X_J^s(\lambda,t_0) }^2.
	\end{equation}
	 Now, we set  $T_\# \in (0,T]$ to be the first time at which $\norm{ u }{ C_{T_\#} X_J^s(\lambda,t_0) } = 2 \norm{ u_0 }{ X_J^s(\lambda,t_0) }$. If $\norm{ u }{ C_T X_J^s(\lambda,t_0) } < 2 \norm{ u_0 }{ X_J^s(\lambda,t_0) }$, we set $T_\# \coloneq T$. If $T_\# = T$, the conclusion follows by \eqref{eqn:I bd} and \eqref{eqn:II bd}, with $T$ replaced by $\min\{T_\ast, T\}$. Otherwise, with $T$ replaced by $T_\#$, still from \eqref{eqn:Duhamel}, \eqref{eqn:I bd} and \eqref{eqn:II bd}, and since $\norm{ S(t) u_0 }{ X_J^s(\lambda,t_0) } \le \norm{ u_0 }{ X_J^s(\lambda,t_0) }$, we have that
	\begin{equation*}
		\begin{aligned}
			& 2 \norm{ u_0 }{ X_J^s(\lambda,t_0) } \\
			& \le \norm{ u_0 }{ X_J^s(\lambda,t_0) } + C_\tau \bigl( \re^{ (\frac\lambda4 + \tau) t_0 J } T_\# + \sqrt{ T_\# } \bigr) ( 2 \norm{ u_0 }{ X_J^s(\lambda,t_0)} )^2 \\
			& \le \norm{ u_0 }{ X_J^s(\lambda,t_0) } + C_\tau \bigl( \re^{ (\frac\lambda4 + \tau) t_0 J } T_\# + \sqrt{ \re^{ (\frac\lambda4 + \tau) t_0 J } T_\# } \bigr) ( 2 \norm{ u_0 }{ X_J^s(\lambda,t_0)} )^2 \\
			& \le \norm{ u_0 }{ X_J^s(\lambda,t_0) } + 2 C_\tau \max \bigl\{ \re^{ (\frac\lambda4 + \tau) t_0 J } T_\# , \sqrt{ \re^{ (\frac\lambda4 + \tau) t_0 J } T_\# } \bigr\} ( 2 \norm{ u_0 }{ X_J^s(\lambda,t_0) } )^2.
		\end{aligned}
	\end{equation*}
	From the above it follows that $ T_\# \gtrsim_\tau \min\{ \norm{ u_0 }{ X_J^s(\lambda,t_0) }^{-1} , \norm{ u_0 }{ X_J^s(\lambda,t_0) }^{-2} \} \re^{-(\frac\lambda4+\tau) t_0 J}$. Choosing $c$ small enough (depending on $C_\tau$, which is bounded uniformly for $t_0$ in compact subsets of $(0,\infty)$, by Lemma~\ref{lem:X prod est}) we deduce that $T_\# \ge T_\ast$, so that $\norm{ u(t) }{ X_J^s(\lambda,t_0) } \le 2 \norm{ u_0 }{ X_J^s(\lambda,t_0) }$ for $0 \le t \le T_\ast$. Inserting this bound into \eqref{eqn:I bd} and \eqref{eqn:II bd}, with $T$ replaced by $T_\ast$, we obtain $\norm{ u(t) - S(t) u_0 }{ X_J^s(\lambda,t_0) } \les_\tau  \norm{ u_0 }{ X_J^s(\lambda,t_0) }^2$, which concludes the proof.
\end{proof}

Fix $\lambda,s,J$. In the rest of this section, $u$ solves \eqref{eqn:NSE} on an interval $[T_0,T]$, and we consider the norm $\norm{ u(t) }{X_J^s(\lambda,t)}$, with $X_J^s(\lambda,t)$ as in \eqref{eqn:X defn}. Note that $t \mapsto \re^{\lambda t j} \norm{ P_j u(t) }{ \dot H^s }^2$ is $C^2$, since $u \in C^2([T_0,T]; \dot H^{\sigma - 4})$ by the equation and $s \le \sigma - 6$. We can compute directly
\begin{align}
	\notag & \frac{\rd}{\rd t} \bigl( \re^{\lambda t j} \norm{ P_j u(t) }{ \dot H^s }^2 \bigr) \\
	\label{eqn:rj'} & = \re^{\lambda t j} \bigl( \lambda j \norm{ P_j u(t)}{ \dot H^s }^2 - 2 \norm{ P_j u(t) }{\dot H^{s+1}}^2 + 2 \innerprod{ P_j u(t) }{ B(u(t)) }{ \dot H^s} \bigr), \\
	\notag & \frac{\rd^2}{\rd t^2} \bigl( \re^{\lambda t j} \norm{ P_j u(t) }{ \dot H^s }^2 \bigr) \\
	\label{eqn:rj''} & = \re^{ \lambda t j } \bigl( (\lambda j)^2 \norm{ P_j u(t)}{ \dot H^s }^2 - 4 \lambda j \norm{ P_j u(t) }{\dot H^{s+1}}^2 + 4 \lambda j \innerprod{ P_j u(t) }{ B(u(t)) }{ \dot H^s } \\
	\notag & \hskip 60pt + 4 \norm{ P_j u(t)}{\dot H^{s+2}}^2 - 6 \innerprod{  P_j u(t) }{ B(u(t)) }{ \dot H^{s+1} } + 2\norm{P_jB(u(t))}{\dot H^s}^2\\
	\notag & \hskip 60pt + 2 \innerprod{ P_j u(t) }{ B\big( \Delta u(t) + B(u(t)), u(t)\big) }{ \dot H^s } \\
	\notag & \hskip 60pt + 2 \innerprod{ P_j u(t) }{ B( u(t), \Delta u(t) + B(u(t)) ) }{ \dot H^s } \bigr) \\
	\notag &= \re^{\lambda t j} \Big( j^2 (\lambda -2)^2 \|P_j u(t) \|_{\dot H^s}^2 + 2\norm{P_jB(u(t))}{\dot H^s}^2 \\
	\notag &\phantom{= \re^{\lambda t j}()} + (4 \lambda - 6) j \innerprod{  P_j u(t) }{ B(u(t)) }{\dot H^s} \\
	\notag &\phantom{= \re^{\lambda t j}()} + 2 \innerprod{ P_j u(t) }{ B\big( \Delta u(t), u(t)\big) + B\big(  u(t), \Delta u(t)\big) }{ \dot H^s } \\
	\notag &\phantom{= \re^{\lambda t j}()} + 2 \innerprod{ P_j u(t) }{ B\big( B( u(t)), u(t)\big) + B\big(  u(t), B( u(t))\big) }{ \dot H^s } \Big).
\end{align}
In the last equality we used that $|n|^2 = j$ for every $n \in A_j$ (as $|n|^2 \in \ZZ$), so that, for example, $\norm{P_ju}{\dot H^{s+1}}^2 = j \norm{P_ju}{\dot H^s}^2$.
We have the following lower bound.
\begin{prop} \label{lem:r'' bd}
	There exist $1 < \lambda_0 < 2$ and $\delta_0 > 0$ such that the following holds. Let $\lambda_0 < \lambda < 2$ and $0 < \delta < \delta_0$. Let $K \ge 1$, let $j \le J$, and let $0 \le T_0 \le t_1 < T$ with $t_1 > 0$. Suppose that $u$ is a solution to \eqref{eqn:NSE} on $[T_0,T]$, such that $\norm{ u }{ C([T_0,T];\dot H^\sigma)} \le K$, and that $\frac d2 < s \le \sigma - 6$. Then there exist $\eta = \eta(\lambda,\delta) \in (0,1)$, $A = A(\lambda,\delta) > 0$ and $C = C(d,s,\lambda,\delta,t_1,T) > 0$, independent of $j$, $J$ and $K$, such that, setting $\lambda_J \coloneq (\lambda - 1 - \eta)J$,
	\begin{equation} \label{eqn:r'' bd}
		\begin{aligned}
			\frac{\rd^2}{\rd t^2} \bigl( \re^{\lambda t j} \norm{ P_j u(t) }{ \dot H^s }^2 \bigr)
			& \ge \re^{\lambda t j} \norm{ P_j B(u(t)) }{ \dot H^s }^2 - j^2 \re^{\lambda t j} \norm{ P_j u(t) }{ \dot H^s }^2 \\
			& \hskip 15pt - \re^{ - \delta tJ} \norm{ B(u(t)) }{ X_J^s(\lambda,t) }^2 - \norm{ u(t) }{ X_J^s(\lambda,t) }^{A} - C K^4 \re^{\lambda_J t}
		\end{aligned}
	\end{equation}
	for all $t \in [t_1,T]$.
\end{prop}
\begin{proof}
	Throughout the proof we write $u$ for $u(t)$.
	Let $\kappa \ll 1$ and $\eta = \eta(\lambda,\delta) \ll 1$ be such that
	\begin{equation} \label{eqn:r'' bd param}
		\tfrac\lambda3 + \kappa + \tfrac\delta4 < \tfrac34 ( \lambda - 1 - \eta ).
	\end{equation}
	This is possible if we take $\lambda_0 < 2$ close enough to $2$, and $\delta_0 \ll 1$, since \eqref{eqn:r'' bd param} holds for $\lambda = 2$ and $\kappa = \delta = \eta = 0$.
	Since $t \ge t_1$, polynomial factors in $J$ can be absorbed into exponentials, $J^p \le C(p,\kappa,t_1) \re^{\kappa t J}$, and we will do so without further comment. Finally, we set $\eps \ll 1$, and we allow the constant $C$ to depend on $d$, $s$, $\lambda$, $\delta$, $t_1$ and $T$.
	We will use Young's inequality in the following form. Let $C', \eps' > 0$, $p > 0$, $0 \le b \le \frac83$ and $0 \le a < \lambda - 1 - \eta$, and let $1 < q \le \frac32$ be such that $q a \le \lambda - 1 - \eta$ (which exists since $a < \lambda - 1 - \eta$). If $A \ge \frac{q p}{q - 1}$, so that $\frac{A}{A - p} \le q$, then Young's inequality with exponents $\frac{A}{A - p}$ and $\frac Ap$ gives
	\begin{equation} \label{eqn:Young1}
		C' K^b \re^{a t J} x^p
		\le \eps' x^A + C \bigl( K^b \re^{a t J} \bigr)^{\frac{A}{A - p}}
		\le \eps' x^A + C K^4 \re^{\lambda_J t}
		\qquad \text{for all } x \ge 0,
	\end{equation}
	since $K \ge 1$, $b q \le 4$ and $a q \le \lambda - 1 - \eta$. Here $C$ depends also on $C'$, $\eps'$, $a$, $b$, $p$ and $A$. Moreover, if $0 \le b \le 2$ and $a + \frac\delta4 < \frac34 ( \lambda - 1 - \eta )$, then, for $A$ large enough, Young's inequality with exponents $4$ and $\frac43$, followed by \eqref{eqn:Young1}, gives
	\begin{equation} \label{eqn:Young2}
		C' K^b \re^{a t J} x^p y^{\frac12}
		\le \eps' \re^{-\delta t J} y^2 + \eps' x^A + C K^4 \re^{\lambda_J t}
		\qquad \text{for all } x, y \ge 0.
	\end{equation}
	Below, we apply \eqref{eqn:Young1} and \eqref{eqn:Young2} with $x = \norm{ u }{ X_J^s(\lambda,t) }$, $y = \norm{ B(u) }{ X_J^s(\lambda,t) }$, $\eps' = \frac\eps2$, and finitely many values of $C'$, $a$, $b$ and $p$, and we fix $A = A(\lambda,\delta)$ large enough for all of them. In each case, the condition on $a$ follows from \eqref{eqn:r'' bd param}.
	By \eqref{eqn:rj''},
	\begin{equation*}
		\begin{aligned}
			\frac{\rd^2}{\rd t^2} \bigl( \re^{\lambda t j} \norm{ P_j u }{ \dot H^s }^2 \bigr)
			& = (2-\lambda)^2 j^2 \re^{\lambda t j} \norm{ P_j u }{ \dot H^s }^2 + 2 \re^{\lambda t j} \norm{ P_j B(u) }{ \dot H^s }^2 \\
			& \hskip 15pt + (4\lambda - 6) j\, \re^{\lambda t j} \innerprod{P_j u}{B(u)}{\dot H^s} + 2 \re^{\lambda tj} \innerprod{P_ju}{N(u)}{\dot H^s},
		\end{aligned}
	\end{equation*}
	where $N(u) \coloneq B(\Delta u, u) + B(u, \Delta u) + B(B(u),u) + B(u,B(u))$.

	By the Cauchy--Schwarz and Young inequalities,
	\[\bigl| (4\lambda - 6) j \re^{\lambda t j} \innerprod{P_j u}{B(u)}{\dot H^s} \bigr| \le \re^{\lambda t j} \norm{ P_j B(u) }{ \dot H^s }^2 + (2\lambda - 3)^2 j^2 \re^{\lambda t j} \norm{ P_j u }{ \dot H^s }^2.\]
	Since $(2 - \lambda)^2 - (2\lambda - 3)^2  \ge -1$ for $1 < \lambda < 2$, we obtain
	\begin{equation*}
		\begin{aligned}
			& (2 - \lambda)^2 j^2 \re^{\lambda t j} \norm{ P_j u }{ \dot H^s }^2 + 2 \re^{\lambda t j} \norm{ P_j B(u) }{ \dot H^s }^2 - \bigl| (4\lambda - 6) j \re^{\lambda t j} \innerprod{P_j u}{B(u)}{\dot H^s} \bigr| \\
		& \ge \re^{\lambda t j} \norm{ P_j B(u) }{ \dot H^s }^2 - j^2 \re^{\lambda t j} \norm{ P_j u }{ \dot H^s }^2,
		\end{aligned}
	\end{equation*}
	and, by the first display, it remains to prove that
	\begin{equation} \label{eqn:N bd}
		2\re^{\lambda tj}\bigl| \innerprod{P_ju}{N(u)}{\dot H^s} \bigr| \le \re^{-\delta t J} \norm{ B(u) }{ X_J^s(\lambda,t) }^2 + \norm{ u }{ X_J^s(\lambda,t) }^{A} + CK^4 \re^{\lambda_J t}.
	\end{equation}
	Let $(f,g)$ be one of $(\Delta u, u)$, $(u,\Delta u)$, $(B(u),u)$, $(u,B(u))$. Decomposing
	\[\mathrm{id} = \sum_{k \le J} P_k + P_{>J},\]
	and using the algebra property, we obtain
	\begin{equation} \label{eqn:N bd fg}
		\begin{aligned}
			& \re^{\lambda tj} \bigl| \innerprod{P_ju}{B(f,g)}{\dot H^s} \bigr| \\
			& \lesssim \re^{\frac\lambda2 t j} \bigl( \re^{\frac\lambda2 t j} \norm{ P_j u }{ \dot H^s } \bigr) \\
			&\phantom{\les \re^{\frac\lambda2 t j}} \times\Bigl( \sum_{\substack{k,\ell \le J \\ (k,\ell)\in\Lambda(j)}} \norm{P_kf}{\dot H^s}\norm{P_\ell g}{\dot H^{s+1}} + \sum_{\ell \le J}\norm{P_{>J}f}{\dot H^s}\norm{P_\ell g}{\dot H^{s+1}} \\
			& \phantom{\les \re^{\frac\lambda2 t j} \times\Bigl(\Bigr)} + \sum_{k\le J}\norm{P_k f}{\dot H^s}\norm{P_{>J}g}{\dot H^{s+1}}
			+ \norm{P_{>J}f}{\dot H^s}\norm{P_{>J}g}{\dot H^{s+1}} \Bigr).
		\end{aligned}
	\end{equation}
	Each sum has $\les J^2$ terms. We bound the possible contributions using
	the following bounds. Let $k \le J$ and $s' \ge 0$. For $(k,\ell) \in \Lambda(j)$, we have
	\begin{equation} \label{eqn:low bds}
		\begin{aligned}
			& \norm{P_k u}{\dot H^{s+s'}} = k^{\frac {s'}2} \re^{-\frac\lambda2 t k} \bigl( \re^{\frac\lambda2 t k} \norm{P_k u}{\dot H^s} \bigr), \\
			& \norm{P_k B(u)}{\dot H^{s+s'}} \le k^{\frac{s'}{2}}\norm{B(u)}{\dot H^s}^\frac12 \norm{P_k B(u)}{\dot H^{s}}^\frac12 \\
			& \hphantom{\norm{P_k B(u)}{\dot H^{s+s'}}} \les k^{\frac {s'}2} K \re^{-\frac\lambda4 t k} \bigl( \re^{\frac\lambda4 t k} \norm{P_k B(u)}{\dot H^s}^{\frac12} \bigr),
		\end{aligned}
	\end{equation}
	where for the last bound we used $\norm{B(u)}{\dot H^s} \lesssim K^2$ by the algebra property. Similarly, by interpolation (Lemma~\ref{lem:basics}\,(ii)) and $\sigma \ge s + 6$, for $0 \le s' \le 3$ we have that
	\begin{equation} \label{eqn:high bds}
		\begin{aligned}
			& \norm{P_{>J} u}{\dot H^{s+s'}} \lesssim \norm{P_{>J} u}{\dot H^s}^{\frac12} \norm{P_{>J} u}{\dot H^{s+2s'}}^{\frac12} \le K^{\frac12} \re^{-\frac\lambda4 t J} \bigl( \re^{\frac\lambda4 t J} \norm{P_{>J} u}{\dot H^s}^{\frac12} \bigr), \\
			& \norm{P_{>J} B(u)}{\dot H^{s+1}} \lesssim K \re^{-\frac\lambda4 tJ} \bigl( \re^{\frac\lambda4 tJ} \norm{P_{>J} B(u)}{\dot H^s}^{\frac12} \bigr),
		\end{aligned}
	\end{equation}
	where in the last bound we used $\norm{B(u)}{\dot H^{s+2}} \lesssim \norm{u}{\dot H^{s+3}}^2 \le K^2$.
	Using \eqref{eqn:low bds}, \eqref{eqn:high bds} and $j \le J$, the first sum in \eqref{eqn:N bd fg} is bounded by
	\begin{align*}
		& \lesssim \sum_{\substack{k,\ell \le J \\ (k,\ell)\in\Lambda(j)}} J^{\frac32} \Bigl( \re^{\frac\lambda2 t (j - k - \ell)} \norm{ u }{ X_J^s(\lambda,t) }^3 \\
		& \hskip 30pt + K \bigl( \re^{\lambda t (\frac j2 - \frac k4 - \frac \ell2)} + \re^{\lambda t (\frac j2 - \frac k2 - \frac \ell4)} \bigr) \norm{ u }{ X_J^s(\lambda,t) }^2 \norm{ B(u) }{ X_J^s(\lambda,t) }^{\frac12} \Bigr) \\
		& \lesssim \re^{(\frac\lambda4 + \kappa) t J} \norm{ u }{ X_J^s(\lambda,t) }^3 + K \re^{(\frac\lambda3 + \kappa) t J} \norm{ u }{ X_J^s(\lambda,t) }^2 \norm{ B(u) }{ X_J^s(\lambda,t) }^{\frac12} \\
		& \le \eps \re^{-\delta tJ}\norm{ B(u) }{ X_J^s(\lambda,t) }^2 + \eps \norm{ u }{ X_J^s(\lambda,t) }^{A} + CK^4\re^{\lambda_J t}
	\end{align*}
	since $(k,\ell) \in \Lambda(j)$ (recall \eqref{eqn:kljinequality} in Lemma~\ref{lem:Lambda}, which gives $j - k - \ell \le \frac j2 + 2$ and $\frac k4 + \frac\ell2, \frac k2 + \frac\ell4 \ge \frac j6 - \frac34$) and by \eqref{eqn:Young1} with $(a,b,p) = (\frac\lambda4 + \kappa, 0, 3)$ for the first term, and by \eqref{eqn:Young2} with $(a,b,p) = (\frac\lambda3 + \kappa, 1, 2)$ for the second term; the conditions on $a$ hold by \eqref{eqn:r'' bd param}, since $\frac\lambda4 + \kappa < \frac\lambda3 + \kappa + \frac\delta4 < \frac34 ( \lambda - 1 - \eta )$.
	Next, if one of $f$ or $g$ is equal to $P_{>J}u$ or $P_{>J}\Delta u$, by \eqref{eqn:high bds} with $\norm{ u }{C_T \dot H^\sigma} \le K$, the total contribution is
	\begin{equation*}
		\begin{aligned}
		& \lesssim K^{\frac52}\re^{(\frac\lambda4+\kappa)tJ}\norm{ u }{ X_J^s(\lambda,t) }^{\frac32} \\
		& \le \eps\re^{-\delta tJ}\norm{ B(u) }{ X_J^s(\lambda,t) }^2 + \eps\norm{ u }{ X_J^s(\lambda,t) }^{A} + CK^4\re^{\lambda_J t}
		\end{aligned}
	\end{equation*}
	by \eqref{eqn:Young1} with $(a,b,p) = (\frac\lambda4 + \kappa, \frac52, \frac32)$.
	The remaining terms are bounded by $\norm{ P_k u }{\dot H^{s + 1}} \norm{ P_{>J} B(u) }{\dot H^{s+1}}$ (note that $\norm{ P_{>J} B(u) }{\dot H^{s}} \le \norm{ P_{>J} B(u) }{\dot H^{s+1}}$) and similarly by \eqref{eqn:high bds} these contribute
	\begin{equation*}
		\begin{aligned}
		& \lesssim K^2\re^{(\frac\lambda4+\kappa)tJ} \norm{ u }{ X_J^s(\lambda,t) } \norm{ B(u) }{ X_J^s(\lambda,t) }^{\frac12} \\
		& \le \eps\re^{-\delta tJ}\norm{ B(u) }{ X_J^s(\lambda,t) }^2 + \eps\norm{ u }{ X_J^s(\lambda,t) }^{A} + CK^4\re^{\lambda_J t}
		\end{aligned}
	\end{equation*}
	by \eqref{eqn:Young2} with $(a,b,p) = (\frac\lambda4 + \kappa, 2, 1)$.

	Summing the contributions, we obtain \eqref{eqn:N bd}, and hence \eqref{eqn:r'' bd}.
\end{proof}

\section{Lower bound and proof of singularity} \label{sec:sing}
In this section we combine the deterministic estimates of Section~\ref{sec:local} with the Gaussian estimates of Section~\ref{sec:gaussian}.
To compare the law of the solution at time $t$ with $\mu_{s_0 + t}$, we will work with the time-translated solution $u(\cdot - s_0)$, which is a solution on $[s_0, s_0 + T]$ (in the sense of Section~\ref{sec:local}).
Let $u$ be such a solution, and suppose that the set $\cT$ of times at which $u(t)$ satisfies the Gaussian predictions of Proposition~\ref{prop:Gaussian} has positive measure. Let $t_0$ be a Lebesgue point of $\cT$. On an interval $[t_0 - h_0, t_0]$ with $h_0 \approx \re^{-(\frac\lambda4 + \tau) t_0 j_k}$, Proposition~\ref{prop:local theory}\,(i) keeps the weighted norm of $u$ bounded, and we can integrate twice the convexity estimate of Proposition~\ref{prop:local theory}\,(ii) with $j = J = j_k$. Since $t \in \cT$ for a positive proportion of times $t \in [t_0 - h_0, t_0 - \frac12 h_0]$, the nonnegative term contributes at least $|A_{j_k}| h_0^2 \re^{(\lambda - 1 - \eta_2) t_0 j_k}$ (up to constants), while the error terms are smaller (Proposition~\ref{prop:lower bound}). Since $\frac\lambda4 < \frac12$, this lower bound is much larger than the bound $M_0^2 |A_{j_k}| \re^{(\lambda - 2 + \eta_0) t_0 j_k}$, which holds at $t_0$ since $t_0 \in \cT$. This contradiction (Corollary~\ref{lem:soln not too smooth}) shows that $\cT$ has measure zero, and Theorem~\ref{thm:main precise} follows by the Fubini argument of Subsection~\ref{subsec:pf pf} (see \eqref{eqn:fub trick}).

\subsection{Lower bound at Lebesgue points} \label{subsec:lower bound}
 Throughout this section, we fix
\begin{equation} \label{eqn:param1}
	0 < \delta < \delta_0, \qquad
	\lambda_0 < \lambda < 2, \qquad
	2 - \lambda < \delta,
\end{equation}
where $\lambda_0$ and $\delta_0$ are given by Proposition~\ref{prop:local theory}\,(ii), and where, increasing $\lambda_0$ if necessary, we assume that $\lambda_0 \ge \frac32$.
We let $\eta = \eta(\lambda,\delta) \in (0,1)$ and $A = A(\lambda,\delta) > 0$ be given by Proposition~\ref{prop:local theory}\,(ii). We then fix $0< \tau, \eta_0, \eta_1, \eta_2 \ll 1$, such that
\begin{equation} \label{eqn:param2}
	2\tau + \eta_0 + 2\eta_1 + 2\eta_2 < \tfrac14 (2 - \lambda), \qquad \eta_2 < \eta,
\end{equation}
where $\eta_0, \eta_1, \eta_2$ are the parameters appearing in \eqref{eqn:G1 defn} -- \eqref{eqn:G defn}
in the definition of the sets $\cG^{(1)}_{t;\vec M,j}, \cG^{(2)}_{t;j}$. We can then apply Proposition~\ref{prop:Gaussian}, and have that
\[		\mu_t ( \cG_t ) = 1  \]
with these parameters.

For $\vec M = (M_0,M_1)$ and $k_0 \ge 1$, let
\begin{equation} \label{eqn:GMk0 defn}
	\cG_{t;{\vec M, k_0}}
	\coloneq \bigcap_{j \ge 1} \cG_{t;\vec M,j}^{(1)} \cap \bigcap_{k \ge k_0} \cG_{t; j_k}^{(2)},
	\end{equation}
so that
	\[\cG_t = \bigcup_{\vec M, k_0} \cG_{t;\vec M,k_0},\]
where the union is increasing in $\vec M$ (componentwise) and in $k_0$. Given $0 \le T_0 < T$ and a solution $u$ to \eqref{eqn:NSE} on $[T_0,T]$, we define the set of times
\begin{equation*}
	\cT_{M_0,M_1,k_0}
	\coloneq \bigl\{ t \in [T_0,T] \colon u(t) \in \cG_{t;\vec M,k_0} \bigr\}.
\end{equation*}
By \eqref{eqn:rj'}, the quantity bounded in the second condition of \eqref{eqn:G1 defn}, evaluated at $u(t)$, is exactly $\bigl| \frac{\rd}{\rd t} \bigl( \re^{\lambda t j} \norm{ P_j u(t) }{ \dot H^s }^2 \bigr) \bigr|$. Therefore, explicitly,
\begin{equation} \label{eqn:good times}
	\begin{aligned}
		\cT_{M_0,M_1,k_0}
		= \bigl\{ t \in [T_0,T] \colon
		& \re^{\lambda t j} \norm{ P_j u(t) }{ \dot H^s }^2 \le M_0^2 |A_j| \re^{(\lambda - 2 + \eta_0) t j}, \\
		& \Bigl| \frac{\rd}{\rd t} \bigl( \re^{\lambda t j} \norm{ P_j u(t) }{ \dot H^s }^2 \bigr) \Bigr| \le M_1^2 |A_j| \re^{(\lambda - \frac32 + \eta_1) t j} \ \text{for all} \ j \ge 1, \\
		& \re^{\lambda t j_k} \norm{ P_{j_k} B(u(t)) }{ \dot H^s }^2 \ge |A_{j_k}| \re^{(\lambda - 1 - \eta_2) t j_k} \ \text{for all} \ k \ge k_0 \bigr\},
	\end{aligned}
\end{equation}
where we recall that $(j_k)$ is the subsequence for which $|A_{j_k}| \ge 1$ and $\sum_k |A_{j_k}|^{-1} < \infty$, as in Lemma~\ref{lem:Aj count}.
Using Proposition~\ref{prop:local theory}, we are ready to prove the following lower bound, which holds at Lebesgue points $t_0$ of the set of times on which the Gaussian estimates of Section~\ref{sec:gaussian} hold. In this subsection, we regard $M_0,M_1$ and $k_0$ as fixed. The main idea is that if $t_0$ is a Lebesgue point of $\cT_{M_0,M_1,k_0}$, there must exist $h_0$ small (of size roughly $T_*(M_0, j)$) such that $t_0-h_0 \in \cT_{M_0,M_1,k_0}$, and so Proposition~\ref{prop:local theory}\,(i) applies to the whole interval $[t_0-h_0,t_0]$. In particular, in view of \eqref{eqn:convexity}, we have that
\[ \re^{\lambda t_0 j} \norm{ P_j u(t_0) }{ \dot H^s }^2 \ge \iint_{[t_0-h_0,t_0]^2} \re^{\lambda t j} \norm{ P_j B(u(t)) }{ \dot H^s }^2 \1_{t \le t'} \, \rd t \, \rd t' + \text{ error terms}.\]
We can then use the fact that $t_0$ is a Lebesgue point of $ \cT_{M_0,M_1,k_0}$ to bound the integrand from below on a large proportion of the interval $[t_0-h_0,t_0]^2$. In particular, we obtain the following.
\begin{prop}[Lower bound at Lebesgue points] \label{prop:lower bound}
	Let $K \ge 1$, $M_0 \ge 1$, $M_1 \ge 0$ and $k_0 \ge 1$, let $0 \le T_0 < T$, and let $u \colon [T_0,T] \times \TT^d \to \RR^d$ be a solution to \eqref{eqn:NSE} on $[T_0,T]$ such that $\norm{ u }{C([T_0,T];\dot H^\sigma)} \le K$. Suppose that $T_0 < t_0 < T$ is a Lebesgue point of $\cT \coloneq \cT_{M_0,M_1,k_0}$. Then there exists a constant $C_0 \ge 1$, depending only on $d$, $\lambda$, $\eta_0$ and $t_0$, and there exist constants $c' > 0$ and $C \ge 1$, independent of $k$, such that, with $T_\ast = T_\ast(C_0 M_0, j_k)$ as in \eqref{eqn:T ast} (where the constant $c$ is chosen uniformly for times in $[\frac{t_0}2, t_0]$), we have
	\begin{equation} \label{eqn:lower bound}
		\begin{aligned}
			\re^{\lambda t_0 j_k} \norm{ P_{j_k} u(t_0) }{ \dot H^s }^2 \ge c' T_\ast^2 \re^{(\lambda - 1 - \eta_2) t_0 j_k} - C T_\ast (M_0^{A+2} + M_1^2) \re^{(\lambda - \frac32 + 2\eta_1) t_0 j_k}
		\end{aligned}
	\end{equation}
	for every $k$ sufficiently large, depending on $K$, $M_0$, $M_1$, $k_0$, $T_0$, $t_0$ and $T$.
\end{prop}
\begin{proof}
	Since $t_0$ is a Lebesgue point of $\cT$, there exists $h_\ast \in (0, \frac{t_0 - T_0}2]$ such that $|\cT \cap [t_0 - h, t_0 - \frac12 h]| \ge \frac14 h$ for all $0 < h \le h_\ast$. Let
	\begin{equation*}
		C_1 \coloneq \max\Bigl\{ 1, \sup_{t' \in [t_0/2,t_0]} \sum_{j \ge 1} |A_j| \re^{(\lambda - 2 + \eta_0) t' j} \Bigr\},
		\qquad
		C_0 \coloneq (8 C_1)^{\frac12},
	\end{equation*}
	where $C_1 < \infty$ since $|A_j| \lesssim j^{\frac d2}$ and $\lambda - 2 + \eta_0 < 0$ by \eqref{eqn:param2}. If $t' \in \cT \cap [\frac{t_0}2, t_0]$, then, by \eqref{eqn:good times}, for every $J \ge 1$ we have
	\begin{equation} \label{eqn:deterministic bd}
		\norm{ u(t') }{ X_J^s(\lambda,t') }^2
		\le \sum_{j \ge 1} \re^{\lambda t' j} \norm{ P_j u(t') }{ \dot H^s }^2
		\le C_1 M_0^2.
	\end{equation}
	For $k$ large, $T_{\ast} \le h_\ast$, so that $| \cT \cap [t_0 - T_{\ast}, t_0 - \frac12 T_{\ast}] | \ge \frac14 T_{\ast} > 0$, and we can choose $h_0 = h_{0,k} \in [\frac12 T_{\ast}, T_{\ast}]$ with $t_0 - h_0 \in \cT$. Note that $t_0 - h_0 \ge \frac{t_0 + T_0}2 \ge \frac{t_0}2$, and that $h_0 \cdot j_k \le T_{\ast} j_k \to 0$ as $k \to \infty$. We apply Proposition~\ref{prop:local theory}\,(i) to $u$ on $[t_0 - h_0, T]$, with $J = j_k$ and with $t_0 - h_0$ in place of $t_0$ in the norm. Since the constant $c$ is the same for all times in $[\frac{t_0}2, t_0]$, and $\re^{-(\frac\lambda4 + \tau)(t_0 - h_0) j_k} \ge \re^{-(\frac\lambda4 + \tau) t_0 j_k}$, the time of existence given by this proposition is at least $T_{\ast} \ge h_0$, by \eqref{eqn:deterministic bd} and since $C_1^{\frac12} \le C_0$. Hence,
	\begin{equation*}
		\norm{ u(t) }{ X_{j_k}^s(\lambda,t_0 - h_0) }
		\le 2 \norm{ u(t_0 - h_0) }{ X_{j_k}^s(\lambda,t_0 - h_0) }
		\le 2 C_1^{\frac12} M_0
		\qquad \text{for all } t \in [t_0 - h_0, t_0].
	\end{equation*}
	For $t \in [t_0 - h_0, t_0]$ and $k$ large, the weights of $X_{j_k}^s(\lambda,t)$ and $X_{j_k}^s(\lambda,t_0 - h_0)$ differ by a factor at most $\re^{\lambda h_0 {j_k}} \le 2$. We conclude that
	\begin{equation} \label{eqn:X bd on interval}
		\norm{ u(t) }{ X_{j_k}^s(\lambda,t) }^2
		\le 8 C_1 M_0^2 = (C_0 M_0)^2 \qquad \text{for all } t \in [t_0 - h_0, t_0].
	\end{equation}
	We apply Proposition~\ref{prop:local theory}\,(ii) with $t_1 = \frac{t_0 + T_0}2$ and $j = J = j_k$, and we integrate the resulting bound \eqref{eqn:convexity} twice from $t_0 - h_0$ to $t_0$, using that $\int_a^b \int_a^t f(r) \, \rd r \, \rd t = \int_a^b (b - t) f(t) \, \rd t$. This gives
	\begin{equation} \label{eqn:r'' bd int'd}
		\begin{aligned}
			\re^{\lambda t_0 j_k} \norm{ P_{j_k} u(t_0) }{ \dot H^s }^2
			& \ge \re^{\lambda (t_0 - h_0) j_k} \norm{ P_{j_k} u(t_0 - h_0) }{ \dot H^s }^2 + \frac{\rd}{\rd t} \bigl( \re^{\lambda t j_k} \norm{ P_{j_k} u(t) }{ \dot H^s }^2 \bigr) \big|_{t = t_0 - h_0} h_0 \\
			& \hskip 15pt + \int_{t_0 - h_0}^{t_0} (t_0 - t) \re^{\lambda t j_k} \norm{ P_{j_k} B(u(t)) }{ \dot H^s }^2 \, \rd t \\
			& \hskip 15pt - \int_{t_0 - h_0}^{t_0} (t_0 - t) \bigl( j_k^2 \re^{\lambda t j_k} \norm{ P_{j_k} u(t) }{ \dot H^s }^2 + \re^{ - \delta t j_k} \norm{ B(u(t)) }{ X_{j_k}^s(\lambda,t) }^2 \\
			& \hskip 60pt + \norm{ u(t) }{ X_{j_k}^s(\lambda,t) }^{A} + C K^4 \re^{ \lambda_{j_k} t }\bigr) \, \rd t.
		\end{aligned}
	\end{equation}
	We estimate the terms in turn. Note $\re^{\lambda (t_0 - h_0) j_k} \norm{ P_{j_k} u(t_0 - h_0) }{ \dot H^s }^2 \ge 0$, and, since $t_0 - h_0 \in \cT$,
	\begin{equation} \label{eqn:deriv lb}
		\frac{\rd}{\rd t} \bigl( \re^{\lambda t j_k} \norm{ P_{j_k} u(t) }{ \dot H^s }^2 \bigr) \big|_{t = t_0 - h_0} h_0 \ge - M_1^2 |A_{j_k}| \re^{(\lambda - \frac32 + \eta_1)(t_0 - h_0)j_k} h_0.
	\end{equation}
	By \eqref{eqn:X bd on interval},
	\[\re^{\lambda t j_k} \norm{ P_{j_k} u(t) }{ \dot H^s }^2 \le \norm{ u(t) }{ X_{j_k}^s(\lambda,t) }^2 \le (C_0 M_0)^2\]
	for $t \in [t_0 - h_0, t_0]$. Since $C_0 M_0 \ge 1$,
	\begin{equation} \label{eqn:error bd}
		\int_{t_0 - h_0}^{t_0} (t_0 - t) \bigl( j_k^2 \re^{\lambda t j_k} \norm{ P_{j_k} u(t) }{ \dot H^s }^2 + \norm{ u(t) }{ X_{j_k}^s(\lambda,t) }^{A} \bigr) \, \rd t
		\le j_k^2 (C_0 M_0)^{A + 2} h_0^2.
	\end{equation}
	For the next two error terms, let $t \in [t_0 - h_0, t_0]$. By Lemma~\ref{lem:Z>E} with $\frac\delta{100}$ in place of $\tau$ and with $[t_1,t_2] = [\frac{t_0}2, t_0]$, since $\norm{ u(t) }{ \dot H^{s+2} } \le K$, and by \eqref{eqn:X bd on interval}, we have
	\[\norm{ B(u(t)) }{ X_{j_k}^s(\lambda,t) }^2 \le C ( M_0^4 + K M_0^3 ) \re^{ (\frac\lambda2 + \frac\delta{100}) t j_k},\]
	where $C$ depends on $d$, $s$, $\delta$, $C_0$ and $t_0$. Hence,
	\begin{equation} \label{eqn:Z survives and Z>E}
		\begin{aligned}
			& \int_{t_0 - h_0}^{t_0} (t_0 - t) \bigl( \re^{-\delta t j_k} \norm{ B(u(t)) }{ X_{j_k}^s(\lambda,t) }^2 + \re^{\lambda_{j_k} t} \bigr) \, \rd t \\
			& \le C ( M_0^4 + K M_0^3 ) h_0^2 \re^{ (\frac\lambda2 + \frac\delta{100}) t_0 j_k - \delta (t_0 - h_0) j_k} + h_0^2 \re^{\lambda_{j_k} t_0}.
		\end{aligned}
	\end{equation}

	For the integral in the second line of \eqref{eqn:r'' bd int'd} whose integrand is nonnegative, since $\lambda - 1 - \eta_2 > 0$, and since $t \in \cT$ implies $\re^{\lambda t j_k} \norm{ P_{j_k} B(u(t)) }{ \dot H^s }^2 \ge |A_{j_k}| \re^{(\lambda - 1 - \eta_2) t j_k}$ for $k \ge k_0$, we have
	\begin{equation} \label{eqn:Zjk lb}
		\begin{aligned}
			\int_{t_0 - h_0}^{t_0} (t_0 - t) \re^{\lambda t j_k} \norm{ P_{j_k} B(u(t)) }{ \dot H^s }^2 \, \rd t
			& \ge |A_{j_k}| \re^{(\lambda - 1 - \eta_2)(t_0 - h_0) j_k} \int_{t_0 - h_0}^{t_0} (t_0 - t) \1_{t \in \cT} \, \rd t \\
			& \ge \tfrac1{8} |A_{j_k}| \re^{(\lambda - 1 - \eta_2)(t_0 - h_0) j_k} h_0^2.
		\end{aligned}
	\end{equation}
	To obtain the first line we used nonnegativity of the integrand, and for the second we integrate on $[t_0 - h_0, t_0 - \frac12 h_0]$, using $| \cT \cap [t_0 - h_0,t_0 - \frac12 h_0] | \ge \frac14 h_0$.
	Inserting $\re^{\lambda (t_0 - h_0) j_k} \norm{ P_{j_k} u(t_0 - h_0) }{ \dot H^s }^2 \ge 0$ and the bounds \eqref{eqn:deriv lb}, \eqref{eqn:error bd}, \eqref{eqn:Z survives and Z>E} and \eqref{eqn:Zjk lb} into \eqref{eqn:r'' bd int'd}, we obtain, for $k$ large, and for constants $c' > 0$ and $C \ge 1$ which do not depend on $k$, and which may change from line to line,
	\begin{align*}
		& \re^{\lambda t_0 j_k} \norm{ P_{j_k} u(t_0) }{ \dot H^s }^2 \\
		& \ge \tfrac18 |A_{j_k}| \re^{(\lambda - 1 - \eta_2)(t_0 - h_0) j_k} h_0^2
		- C h_0^2 \bigl( \re^{ (\frac\lambda2 + \frac\delta{100}) t_0 j_k - \delta (t_0 - h_0) j_k} + \re^{\lambda_{j_k} t_0} \bigr) \\
		& \hskip 15pt - M_1^2 |A_{j_k}| \re^{(\lambda - \frac32 + \eta_1)(t_0 - h_0)j_k} h_0 - j_k^2 (C_0 M_0)^{A + 2} h_0^2 \\
		& \ge c' T_\ast^2 \re^{(\lambda - 1 - \eta_2) t_0 j_k} - C T_\ast (M_0^{A+2} + M_1^2) \re^{(\lambda - \frac32 + 2\eta_1) t_0 j_k},
	\end{align*}
	where we used that $\frac12 T_\ast \le h_0 \le T_\ast$, that $h_0 j_k \to 0$ and $h_0 \le 1$, and
	\[ \lambda - 1 - \eta_2 > \max \bigl\{\tfrac \lambda 2 - \tfrac{99}{100} \delta, \lambda - 1 - \eta \bigr\}, \qquad \lambda - \tfrac32 + 2\eta_1 > 0,  \]
	which follow from the choice of parameters in \eqref{eqn:param1}--\eqref{eqn:param2}, since $\lambda > \lambda_0 \ge \frac32$, and the fact that $1 \le |A_{j_k}| \les |\{ n \in \ZZ^d: |n| \le 2j_k^2\}|$, which is polynomial in $j_k$. This proves \eqref{eqn:lower bound}.
\end{proof}
As a consequence, solutions to \eqref{eqn:NSE} cannot be too smooth too often. This is key to proving Theorem~\ref{thm:main precise}.
\begin{cor} \label{lem:soln not too smooth}
	Let $0 \le T_0 < T$, and  $u \colon [T_0,T] \times \TT^d \to \RR^d$ be a solution to \eqref{eqn:NSE} on $[T_0,T]$, and let $M_0 \ge 1$, $M_1 \ge 0$ and $k_0 \ge 1$. Then the set of times $\cT_{M_0,M_1,k_0}$ defined in \eqref{eqn:good times} cannot contain any of its Lebesgue points, and therefore must have zero Lebesgue measure.
\end{cor}
\begin{proof}
	Assume, for the sake of contradiction, that $\cT \coloneq \cT_{M_0,M_1,k_0}$ contains one of its Lebesgue points, say $t_0 \in \cT \cap (T_0,T)$. Let $K \coloneq \max\{1, \norm{ u }{ C([T_0,T];\dot H^\sigma) }\}$, and let $C_0$, $c'$, $C$ and $T_\ast = T_\ast(C_0 M_0, j_k)$ be as in Proposition~\ref{prop:lower bound}. By \eqref{eqn:good times}, $\re^{\lambda t_0 j_k} \norm{ P_{j_k} u(t_0) }{ \dot H^s }^2 \le M_0^2 |A_{j_k}| \re^{(\lambda - 2 + \eta_0)t_0 {j_k}}$. Together with Proposition~\ref{prop:lower bound}, this implies
	\begin{equation*}
		M_0^2 |A_{j_k}| \re^{ (\lambda - 2 + \eta_0) t_0 j_k}
		\ge c' T_\ast^2 \re^{(\lambda - 1 - \eta_2) t_0 j_k} - C T_\ast (M_0^{A+2} + M_1^2) \re^{(\lambda - \frac32 + 2\eta_1) t_0 j_k}
	\end{equation*}
	for $k$ sufficiently large. By \eqref{eqn:T ast}, $T_\ast = c_\ast \re^{-(\frac\lambda4 + \tau) t_0 j_k}$, where $c_\ast \coloneq c (1 + C_0 M_0)^{-2} > 0$ does not depend on $k$. Therefore, we obtain
	\[M_0^2 |A_{j_k}| \re^{ (\lambda - 2 + \eta_0) t_0 j_k}
		\ge c' \re^{(\lambda/2 - 1 - \eta_2 -2\tau) t_0 j_k} - C (M_0^{A+2} + M_1^2) \re^{(\frac34 \lambda - \frac32 + 2\eta_1 -\tau) t_0 j_k},
 \]
 for some different $c', C > 0$.
 By taking $k \to \infty$, and since $|A_{j_k}| \lesssim j_k^{\frac d2}$, we obtain that
 \[\lambda/2 - 1 - \eta_2 -2\tau \le \max \bigl\{ \lambda - 2 + \eta_0, \tfrac34 \lambda - \tfrac32 + 2\eta_1 -\tau \bigr\}, \]
 which is a contradiction with our choice $\lambda < 2$ and $\eta_0, \eta_1, \eta_2, \tau$ small enough in \eqref{eqn:param2}.
\end{proof}

\subsection{Proof of singularity} \label{subsec:pf pf}
In this subsection, we assemble the results of the previous sections to prove our main result and Corollary~\ref{cor:extension}.
We first point out that we will use several measurability properties within the proofs of Theorem~\ref{thm:main precise} and Corollary~\ref{cor:extension}, which we collect here. For $u_0 \in \dot H^\sigma(\TT^d)$, let $T^\ast(u_0) \coloneq \sup \{ t \ge 0 \colon u_0 \in \WP(t) \} \in (0,\infty]$ denote the maximal existence time, and recall the sets $\cG_{t;\vec M,k_0}$ defined in \eqref{eqn:GMk0 defn}, and $\cG_t = \bigcup_{\vec M, k_0} \cG_{t;\vec M,k_0}$. We use the following facts.
\begin{enumerate}[label=(M\arabic*)]
	\item \label{item:M1} $\dot H^\sigma(\TT^d)$ is a Borel subset of $\cD'$, and its Borel $\sigma$-algebra coincides with the trace on $\dot H^\sigma$ of the Borel $\sigma$-algebra of $\cD'$. Moreover, for every $t > 0$, $\vec M$ and $k_0$, the sets $\cG_{t;\vec M,k_0}$ and $\cG_t$ are Borel subsets of $\cD'$.
	\item \label{item:M3} For every $t \ge 0$, the set $\WP(t)$ is a Borel subset of $\cD'$, and the map $\Phi_t \colon \WP(t) \to \dot H^\sigma$ is Borel measurable. In particular, $\nu_t = (\Phi_t)_\#(\1_{\WP(t)} \, \nu_0)$ is a well-defined Borel measure on $\cD'$.
	\item \label{item:M4} For every $s_0 \ge 0$, $\vec M$ and $k_0$, the set
	\begin{equation*}
		E_{s_0;\vec M,k_0}
		\coloneq \bigl\{ (t,u_0) \in (0,\infty) \times \cD' \colon u_0 \in \WP(t), \ \Phi_t(u_0) \in \cG_{s_0 + t;\vec M,k_0} \bigr\}
	\end{equation*}
	is measurable with respect to the product $\sigma$-algebra $\mathcal{B}((0,\infty)) \otimes \mathcal{B}(\cD')$.
	\item \label{item:M5} For every $s_0 \ge 0$, $\vec M$ and $k_0$, the maps $t \mapsto \nu_t(\cG_{s_0 + t;\vec M,k_0})$ and $t \mapsto \nu_t(\cG_{s_0 + t})$ are Borel measurable on $(0,\infty)$.
\end{enumerate}
All of the above are well-known, or standard consequences of the locally uniform continuity of the flow up to the time of existence, together with the blowup alternative
\[ \lim_{t \uparrow T^*} \|u(t)\|_{\dot H^\sigma}  = \infty\]
if $T^* < \infty$ (see e.g.\ \cite{RRS16} for $d = 3$, and the discussion in Subsection~\ref{subsec:notation}), so we will use them without further comments.
\begin{proof}[Proof of Theorem~\ref{thm:main precise}]
	Fix $s_0 \ge 0$ and $T > 0$; it suffices to show that $|\cN_{s_0} \cap (0,T)| = 0$. Let $t \in \cN_{s_0}$. Since $\mu_{s_0 + t}(\cG_{s_0 + t}) = 1$ by Proposition~\ref{prop:Gaussian}, and $\nu_t \not\perp \mu_{s_0+t}$, we have $\nu_t(\cG_{s_0 + t}) > 0$. Therefore $\cN_{s_0} \subseteq \{ t > 0 \colon \nu_t(\cG_{s_0 + t}) > 0 \}$. Consequently, it is enough to show that
	\begin{equation*}
		\int_0^T \nu_t(\cG_{s_0 + t}) \, \rd t = 0,
	\end{equation*}
	where the integrand is measurable by \ref{item:M5}. Since $\cG_{s_0 + t}$ is the increasing union of the sets $\cG_{s_0 + t;\vec M,k_0}$, by monotone convergence it suffices to show that $\int_0^T \nu_t(\cG_{s_0 + t;\vec M,k_0}) \, \rd t = 0$ for every fixed $\vec M = (M_0,M_1)$ with $M_0 \ge 1$ and every fixed $k_0 \ge 1$. By definition of $\nu_t$, \ref{item:M4} and Tonelli's theorem,
	\begin{equation} \label{eqn:fub trick}
		\begin{aligned}
			\int_0^T \nu_t \bigl( \cG_{s_0 + t;\vec M,k_0} \bigr) \, \rd t
			& = \int_0^T \int \1_{\WP(t)}(u_0) \, \1_{\cG_{s_0 + t;\vec M,k_0}} \bigl( \Phi_t(u_0) \bigr) \, \rd \nu_0(u_0) \, \rd t \\
			& = \int \Bigl| \bigl\{ t \in (0, \min\{T, T^\ast(u_0)\}) \colon \Phi_t(u_0) \in \cG_{s_0 + t;\vec M,k_0} \bigr\} \Bigr| \, \rd \nu_0(u_0),
		\end{aligned}
	\end{equation}
	where we used that $\{ t \colon u_0 \in \WP(t) \} = [0, T^\ast(u_0))$. Therefore, it suffices to show that, for every $u_0 \in \dot H^\sigma$ and every $0 < T' < T^\ast(u_0)$, the set
	\begin{equation*}
		\cT' \coloneq \bigl\{ t \in [0,T'] \colon \Phi_t(u_0) \in \cG_{s_0 + t;\vec M,k_0} \bigr\}
	\end{equation*}
	has measure zero; the claim then follows by letting $T' \uparrow \min\{T, T^\ast(u_0)\}$. Fix such $u_0$ and $T'$, and let $v(r) \coloneq \Phi_{r - s_0}(u_0)$ for $r \in [s_0, s_0 + T']$. Since \eqref{eqn:NSE} is autonomous, $v$ is a solution to \eqref{eqn:NSE} on $[s_0, s_0 + T']$ in the sense of Section~\ref{sec:local}. Since $v(s_0 + t) = \Phi_t(u_0)$, the set $s_0 + \cT'$ is the set $\cT_{M_0,M_1,k_0}$ of \eqref{eqn:good times} associated with the solution $v$ on $[T_0,T] = [s_0, s_0 + T']$. By Corollary~\ref{lem:soln not too smooth}, $|s_0 + \cT'| = 0$, and so $|\cT'| = 0$. This concludes the proof.
\end{proof}
Before showing Corollary~\ref{cor:extension}, we need the following lemma.
\begin{lemma} \label{lem:J countable}
Let $\cE$ and $\tilde\Phi$ be as in Corollary~\ref{cor:extension}, let $u_0 \in \cE$, and define the set of times

\begin{equation*}
	\begin{aligned}
		\mathscr J(u_0) \coloneq \bigl\{ t > 0 \colon
		& \tilde\Phi_s(u_0) \notin \WP(t - s) \ \text{for every} \ 0<s<t, \\
		& \text{and there is} \ r > t  \ \text{such that} \\
		& \tilde\Phi_s(u_0) \in \WP(r - s) \ \text{for every} \ t < s < r \bigr\}.
	\end{aligned}
\end{equation*}
That is, the set of times $t$ such that the evolution of $u_0$ under $\tilde\Phi$ is given by the smooth flow $\Phi$ on an interval $(t,r]$, but not on any interval $[t - h, t]$, $h > 0$. Then $\mathscr J(u_0)$ is countable. Moreover, if $t > 0$ is such that $\tilde\Phi_t(u_0) \in \dot H^\sigma$ and $\tilde\Phi_s(u_0) \notin \WP(t - s)$ for every $s \in (0,t)$, then $t \in \mathscr J(u_0)$.
\end{lemma}
\begin{proof}
 For every $t \in \mathscr J(u_0)$, we can pick $q_t \in \QQ$, $q_t > t$ such that the properties of the definition hold with $q_t$ replacing $r$.   Therefore, we can define the map $\iota: \mathscr J(u_0) \to \QQ$ given by $\iota(t) = q_t$. If $t'>t$ are elements of $\mathscr J$, then we must have $t' > q_t$, since otherwise $t' \in (t, q_t]$, and, choosing $t < s < t'$, we would have $\tilde\Phi_s(u_0) \in \WP(q_t - s) \subseteq \WP(t' - s)$, which contradicts the definition of $\mathscr J$. Therefore, the map $\iota$ is increasing on $\mathscr J$, hence injective. Therefore, $\mathscr J$ is countable.

For the second claim, let $v \coloneq \tilde\Phi_t(u_0) \in \dot H^\sigma$. In view of the standard local well-posedness for \eqref{eqn:NSE} in $\dot H^\sigma$, we have $T^\ast(v) > 0$.  Let $r \in (t, t + T^\ast(v))$. If $s \in (t,r)$, then $v \in \WP(r - t)$, so that, by \eqref{eqn:consistency} and the uniqueness of mild solutions, $\tilde\Phi_s(u_0) = \Phi_{s - t}(v) \in \WP(r - s)$. Hence, $t \in \mathscr J(u_0)$.
\end{proof}
Now we are ready to prove Corollary~\ref{cor:extension}.
\begin{proof}[Proof of Corollary~\ref{cor:extension}]
	The argument is analogous to the implication from Proposition 1.6 to Theorem 1.1 in \cite{CT26}.  For $0 < s < t$, let $\cR_{s,t} \coloneq \tilde\Phi_s^{-1}\bigl( \WP(t - s) \bigr)$, i.e.\ the set of initial data $u_0$ such that $\tilde \Phi_s(u_0)$ can be further extended to time $t$ using the smooth flow $\Phi$. By \ref{item:M3}, $\WP(t - s)$ is a Borel subset of $\cD'$, so $\cR_{s,t}$ is a Borel set. If $u_0 \in \cR_{s,t}$ and $s < s' < t$, then, by \eqref{eqn:consistency} and the uniqueness of mild solutions, $\tilde\Phi_{s'}(u_0) = \Phi_{s' - s}(\tilde\Phi_s(u_0)) \in \WP(t - s')$, that is, $u_0 \in \cR_{s',t}$. Therefore, the set $\cR_t \coloneq \bigcup_{s \in (0,t)} \cR_{s,t} = \bigcup_{s \in \QQ \cap (0,t)} \cR_{s,t}$ is a Borel set.
Informally, this is the set of initial data $u_0$ such that the solution $u$ is smooth on a short interval $[t-h, t]$, where $h$ is allowed to depend on $u_0$. Moreover, if $u_0 \in \cR_{s,t}$, then $T^\ast(\tilde\Phi_s(u_0)) > t - s$, since $\Phi_{t-s}(\tilde\Phi_s(u_0)) \in \dot H^\sigma$ can be further evolved by the local theory, and so $u_0 \in \cR_{s,q}$ for every $q \in (t, s + T^\ast(\tilde\Phi_s(u_0)))$. Hence,
	\begin{equation} \label{eqn:Rt rect}
		\cR_t = \bigcup_{s, q \in \QQ, \ 0 < s < t < q} \cR_{s,q},
	\end{equation}
	and in particular the set $\{ (t,u_0) \colon t > 0, \ u_0 \in \cR_t \} = \bigcup_{s,q \in \QQ, \ 0 < s < q} (s,q) \times \cR_{s,q}$ is measurable with respect to $\mathcal{B}((0,\infty)) \otimes \mathcal{B}(\cD')$.

	Now fix $s \in \QQ_{>0}$, and let $\rho_s \coloneq (\tilde\Phi_s)_\# \bigl( \1_{\tilde\Phi_s^{-1}(\dot H^\sigma)} \, \nu_0 \bigr)$, which is well defined by \ref{item:M1}. This is a finite measure, concentrated on divergence-free fields in $\dot H^\sigma$. If $\rho_s \ne 0$, Theorem~\ref{thm:main precise}, applied with initial distribution $\rho_s / \rho_s(\cD')$ and with $s_0 = s$, gives a Lebesgue null set $\cN'_s \subseteq (0,\infty)$ such that $(\Phi_r)_\#(\1_{\WP(r)} \, \rho_s) \perp \mu_{s + r}$ for every $r \in (0,\infty) \setminus \cN'_s$; if $\rho_s = 0$, we let $\cN'_s = \emptyset$. By \eqref{eqn:consistency}, for every $t > s$ and every Borel set $A \subseteq \cD'$,
	\begin{equation*}
		\begin{aligned}
			(\Phi_{t-s})_\#(\1_{\WP(t-s)} \, \rho_s)(A)
			& = \nu_0 \bigl( \bigl\{ u_0 \in \cR_{s,t} \colon \Phi_{t-s}\bigl( \tilde\Phi_s(u_0) \bigr) \in A \bigr\} \bigr) \\
			& = (\tilde\Phi_{t})_\#(\1_{\cR_{s,t}} \, \nu_0)(A).
		\end{aligned}
	\end{equation*}
	Hence, $(\tilde\Phi_t)_\#(\1_{\cR_{s,t}} \, \nu_0) \perp \mu_t$ for every $t \in (s,\infty) \setminus (s + \cN'_s)$.
Since
	\[\bigcup_{s\in \QQ\cap(0,t)} \cR_{s,t} = \cR_t,\]
	then, for every
	\[t \not\in \bigcup_{s\in \QQ\cap(0,t)} s + \cN'_s,  \]
	we have that $(\tilde\Phi_t)_\#(\1_{\cR_t} \nu_0) \perp \mu_t$.
	Moreover, since $\mu_t(\dot H^
	\sigma) = 1$ for $t > 0$ (see Section~\ref{sec:gaussian}), we have
	\[(\tilde\Phi_t)_\#((\1-\1_{\tilde \Phi_t^{-1}(\dot H^\sigma)}) \nu_0) \perp \mu_t\]
	for every $t > 0$.
It remains to consider the Borel set (see \ref{item:M1})
	\begin{equation*}
		\cS_t \coloneq \tilde\Phi_t^{-1}(\dot H^\sigma) \setminus \cR_t
	\end{equation*}
	of data whose evolution is in $\dot H^\sigma$ at time $t$, but is not smooth on any interval $[t - h, t]$, $h > 0$. For $t > 0$, let
	\begin{equation*}
		\mathscr J_t
		\coloneq \bigl\{ u_0 \in \cE \colon t \in \mathscr J(u_0) \bigr\}
		= \bigcup_{q \in \QQ \cap (t,\infty)} \Bigl( \bigcap_{s \in \QQ \cap (t,q)} \cR_{s,q} \Bigr) \setminus \cR_t.
	\end{equation*}
	Here, in the second equality, we used that, if $\tilde\Phi_s(u_0) \in \WP(r - s)$ for every $s \in (t,r)$, then the same holds with $r$ replaced by any $q \in (t,r)$, and that $\cR_{s',q} \subseteq \cR_{s,q}$ for $s' < s < q$, so that it suffices to consider rational $q$ and $s$. By Lemma~\ref{lem:J countable}, $\cS_t \subseteq \mathscr J_t$. Moreover, the set $\{ (t,u_0) \colon t > 0, \ u_0 \in \mathscr J_t \}$ is obtained from the sets $(0,s) \times \cD'$, $(0,q) \times \cD'$, $(0,\infty) \times \cR_{s,q}$, $s, q \in \QQ$, and $\{ (t,u_0) \colon t > 0, \ u_0 \in \cR_t \}$ by countably many unions, intersections and differences, so it is measurable with respect to $\mathcal{B}((0,\infty)) \otimes \mathcal{B}(\cD')$ by \eqref{eqn:Rt rect}, and its sections $\mathscr J(u_0)$ are countable by Lemma~\ref{lem:J countable}. Therefore, by Tonelli's theorem,
	\begin{equation*}
		\int_0^\infty \nu_0(\mathscr J_t) \, \rd t
		= \int \bigl| \mathscr J(u_0) \bigr| \, \rd \nu_0(u_0)
		= 0,
	\end{equation*}
	and so $\nu_0(\cS_t) \le \nu_0(\mathscr J_t) = 0$ for almost every $t > 0$.

	Finally, by \eqref{eqn:consistency}, $\cR_t \subseteq \tilde\Phi_t^{-1}(\dot H^\sigma)$, so that $\cE$ is the disjoint union of $\cR_t$, $\cS_t$ and $\cE \setminus \tilde\Phi_t^{-1}(\dot H^\sigma)$. For almost every $t > 0$ we have $t \notin \bigcup_{s \in \QQ_{>0}} (s + \cN'_s)$ and $\nu_0(\cS_t) = 0$, and therefore
	\begin{equation*}
		(\tilde\Phi_t)_\# \nu_0
		= (\tilde\Phi_t)_\# ( \1_{\cR_t} \, \nu_0 ) + (\tilde\Phi_t)_\# \bigl( (\1 - \1_{\tilde\Phi_t^{-1}(\dot H^\sigma)}) \, \nu_0 \bigr)
	\end{equation*}
	is the sum of two measures which are singular with respect to $\mu_t$. Hence, $(\tilde\Phi_t)_\# \nu_0 \perp \mu_t$ for almost every $t > 0$, which concludes the proof.

\end{proof}

\end{document}